\documentclass[12pt]{amsart}

\usepackage{a4}
\usepackage{amsmath}
\usepackage{amssymb,verbatim}
\usepackage{amscd}

\newtheorem{propo}{Proposition}[section]
\newtheorem{corol}[propo]{Corollary}
\newtheorem{theor}[propo]{Theorem}
\newtheorem{lemma}[propo]{Lemma}
\theoremstyle{definition}
\newtheorem{defin}[propo]{Definition}
\newtheorem{examp}[propo]{Example}
\newtheorem{examps}[propo]{Examples}
\newtheorem{assum}[propo]{Assumption}
\theoremstyle{definition}
\newtheorem{remar}[propo]{Remark}
\newtheorem{question}[propo]{Question}
\newtheorem{oppro}[propo]{Open Problems}
\newtheorem{conve}[propo]{Convention}

\numberwithin{equation}{section}

\newcommand{\ad }{\mathrm{ad}\,}

\newcommand{\Aut }{\mathrm{Aut}}              

\newcommand{\cC }{\mathcal{C}}

\newcommand{\cM }{\mathcal{M}}

\newcommand{\co }{\mathrm{co}\,}              

\newcommand{\copr }{\varDelta }               

\newcommand{\fie }{\Bbbk }

\newcommand{\Hom }{\mathrm{Hom}}

\newcommand{\id }{\mathrm{id}}

\newcommand{\lact }{\cdot }

\newcommand{\NA }{\mathcal{B}}
\newcommand{\ndN }{\mathbb{N}}
\newcommand{\ndZ }{\mathbb{Z}}

\newcommand{\op }{\mathrm{op}}
\newcommand{\ot }{\otimes }

\newcommand{\V }[1]{V^{#1}}

\newcommand{\ydH }{ {}^{\phantom{.}H}_H\mathcal{YD}}

\renewcommand{\_}[1]{_{(#1)}}
\renewcommand{\^}[1]{^{(#1)}}

\newcommand{\sub}{\subseteq}
\newcommand{\no}{\mathbb{N}_0}

\newcommand{\res}{\mid}

\newcommand{\cS}{\mathcal{S}}

\newcommand{\sw}[1]{_{(#1)}}
\newcommand{\swo}[1]{^{(#1)}}

\newcommand{\rat}{\mathrm{rat}}

\newcommand{\mZ}{\mathbb{Z}}

\newcommand{\rydH }{\mathcal{YD}^{\phantom{.}H}_H}

\newcommand{\swe}[1]{_{[#1]}}
\newcommand{\sws}[1]{_{\langle#1\rangle}}
\newcommand{\swoe}[1]{^{[#1]}}
\newcommand{\swos}[1]{^{\langle#1\rangle}}

\newcommand{\swoss}[1]{^{\langle\langle#1\rangle\rangle}}

\usepackage{xy}
\newcommand{\ydA}{ {}^{\phantom{.}{A}}_{A}\mathcal{YD}}
\newcommand{\ydB}{ {}^{\phantom{.}{B}}_{B}\mathcal{YD}}
\newcommand{\ydBC}{ {}^{\phantom{.}C}_B\mathcal{YD}}

\newcommand{\rydC}{\mathcal{YD}^{\phantom{.}C}_C}
\newcommand{\rydBC}{\mathcal{YD}^{\phantom{.}C}_B}

\newcommand{\ydRH}{ {}^{\phantom{}R \#H}_{\phantom{aa.}H}\mathcal{YD}}
\newcommand{\ydBNH}{ {}^{\phantom{.}\NA (N)\#H}_{\NA (N)\#H}\mathcal{YD}}
\newcommand{\cZ}{\mathcal{Z}}
\newcommand{\cL}{\mathcal{L}}
\newcommand{\cE}{\mathcal{E}}
\newcommand{\cD}{\mathcal{D}}
\newcommand{\cF}{\mathcal{F}}
\newcommand{\cR}{\mathcal{R}}
\newcommand{\Fd}[1]{\mathcal{F}^{\delta}_{#1}}
\newcommand{\Fm}[1]{\mathcal{F}^{\mu}_{#1}}

\newcommand{\ydAH}{ {}^{\phantom{.}H}_A\mathcal{YD}}

\newcommand{\dual}{\mathrm{D}}

\newcommand{\uydRH}{{}_{R \#H}^{\phantom{aaa.}H}\mathcal{YD}}
\newcommand{\rydRsmash}{\mathcal{YD}^{\phantom{.}R \# H}_{R \# H}}
\newcommand{\rydRsmashrat}{{}_{\rat}{\mathcal{YD}^{\phantom{.}R \# H}_{R \# H}}}
\newcommand{\ydRsmash}{ {}^{\phantom{.}R \# H}_{R \# H}\mathcal{YD}}
\newcommand{\ydRsmashrat}{ {}^{\phantom{.}R \# H}_{R \# H}\mathcal{YD}_{\rat}}
\newcommand{\ydRveesmash}{ {}^{\phantom{.}R^{\vee} \# H}_{R^{\vee} \# H}\mathcal{YD}}
\newcommand{\ydRveesmashrat}{ {}^{\phantom{.}R^{\vee} \# H}_{R^{\vee} \# H}\mathcal{YD}_{\rat}}

\newcommand{\rydRH}{\mathcal{YD}^{\phantom{.}R \# H}_H}
\newcommand{\ydRveeH}{{{}^{\phantom{aaaaa.}H}_{\phantom{a}R^{\vee} \# H}\mathcal{YD}}}

\newcommand{\ydBNdualsmashrat}{ {}^{\phantom{.}\NA(N^*) \# H}_{\NA(N^*) \# H}\mathcal{YD}_{\rat}}
\newcommand{\ydBNsmashrat}{ {}^{\phantom{.}\NA(N) \# H}_{\NA(N) \# H}\mathcal{YD}_{\rat}}

\newcommand{\ydRveeHrat}{{{}^{\phantom{aaaaa.}H}_{\phantom{a}R^{\vee} \# H}\mathcal{YD}_{\rat}}}

\newcommand{\ga}{\gamma}
\newcommand{\la}{\lambda}
\newcommand{\si}{\sigma}

\newcommand{\ph}{\varphi}

\newcommand{\ti}[1]{{\widetilde{#1}}}

\newcommand{\restr}{\mathrm{res}}
\newcommand{\Fu}{\Omega}
\newcommand{\fu}{\omega}

\title[Yetter-Drinfeld modules]{Yetter-Drinfeld modules over bosonizations of dually paired Hopf algebras}

\author{I. Heckenberger}
\address{Istv\'an Heckenberger,
Philipps-Iniversit\"at Marburg,
Fachbereich Matyhematik und Informatik,
Hans-Meerwein-Strasse
D-35032 Marburg, Germany}
\email{heckenberger@mathematik.uni-marburg.de}
\author{H.-J. Schneider}
\address{
Hans-J\"urgen Schneider,
Mathematisches Institut,
Universit\"at M\"unchen,
Theresienstr. 39,
D-80333 Munich, Germany}
\email{Hans-Juergen.Schneider@mathematik.uni-muenchen.de}
\thanks{The work of I.H. is supported by DFG within the Heisenberg program of DFG}

\begin{document}

\begin{abstract}
Let $(R^{\vee},R)$ be a dual pair of Hopf algebras in the category of Yetter-Drinfeld modules over
a Hopf algebra $H$ with bijective antipode. We show that there is a braided monoidal isomorphism
between rational left Yetter-Drinfeld modules over the bosonizations of $R$ and of $R^{\vee}$, respectively.
As an application of this very general category isomorphism we obtain a natural proof of the existence
of reflections of Nichols algebras of semisimple Yetter-Drinfeld modules over $H$.
\end{abstract}

\maketitle

\section*{Introduction}

Let $H$ be a Hopf algebra with bijective antipode over the base field $\fie$, and let $(R^{\vee},R)$
together with a bilinear form $\langle \;,\;\rangle : R^{\vee} \ot R \to \fie$ be a dual pair of
Hopf algebras in the braided category $\ydH$ of left Yetter-Drinfeld modules over $H$
(see Definition \ref{defin:pair}).
The smash products or bosonizations $R^{\vee} \# H$ and $R \# H$ are Hopf algebras in the usual sense.
We are interested in their braided monoidal  categories of left Yetter-Drinfeld modules.
By our first main result, Theorem \ref{theor:third}, there is a braided monoidal isomorphism
\begin{align}\label{intro1}
(\Fu,\fu) : \ydRsmashrat \to \ydRveesmashrat,
\end{align}
where the index rat means Yetter-Drinfeld modules which are rational over $R$ and over $R^{\vee}$
(see Definition \ref{defin:pair}). In particular, $(\Fu,\fu)$ maps Hopf algebras to Hopf algebras.
For $X \in \ydRsmashrat$, $\Fu(X)=X$ as a Yetter-Drinfeld module over $H$.

The origin of the
isomorphism \eqref{intro1} is the standard correspondence between comodule structures over a
coalgebra and module structures over the dual algebra. In Theorem \ref{theor:first} we first prove a monoidal isomorphism between right and left relative Yetter-Drinfeld modules, and hence a braided monoidal isomorphism between their Drinfeld centers. Then we show in Theorem \ref{theor:second} that this isomorphism preserves the subcategories of right and left Yetter-Drinfeld modules we want. Finally, in Theorem \ref{theor:third} we change the sides to left Yetter-Drinfeld modules on both sides. Without this strategy, it would be hard to guess and to prove the correct formulas.

Our motivation to find such an isomorphism of categories comes from the theory
of Nichols algebras which in turn are fundamental for the classification of
pointed Hopf algebras. If $M \in \ydH$, the Nichols algebra $\NA(M)$ is a
braided Hopf algebra in $\ydH$ which is the unique graded quotient of the
tensor algebra $T(M)$ such that $M$ coincides with the space of primitive
elements in $\NA(M)$.

A basic construction to produce new Nichols algebras is the reflection of
semisimple Yetter-Drinfeld modules $M_1 \oplus \cdots \oplus M_\theta$,
where $\theta \in \ndN $ and $M_1,\dots,M_\theta$ are finite-dimensional
and irreducible objects in $\ydH$. For $1 \leq i \leq \theta$, the $i$-th
reflection of $M =(M_1,\dots,M_\theta)$ is a certain $\theta$-tuple
$R_i(M) =(V_1,\dots,V_\theta)$ of finite-dimensional irreducible
Yetter-Drinfeld modules in $\ydH$. It is defined assuming a growth condition
of the adjoint action in the Nichols algebra $\NA(M)$ of
$M_1 \oplus \cdots \oplus M_\theta$. The Nichols algebras
$\NA(R_i(M))$ of $V_1 \oplus \cdots \oplus V_\theta$ and $\NA(M)$ have the
same dimension. The reflection operators allow to define the Weyl groupoid
of $M$, an important combinatorial invariant. In this paper we give a natural
explanation of the reflection operators in terms of the isomorphism
$(\Fu,\fu)$.

To describe our new approach to the reflection operators, fix $1 \leq i \leq \theta$, and let $K_i^M$ be the algebra of right coinvariant elements of $\NA(M)$ with respect to the canonical projection $\NA(M) \to \NA(M_i)$ coming from the direct sum decomposition of $M$. By the theory of bosonization of Radford-Majid, $K_i^M$ is a Hopf algebra in $\ydRsmash$. To define $R_i(M)$ we have to assume that $K_i^M$ is rational as an $R$-module. Let $W= \ad \NA(M_i)(\oplus_{j \neq i} M_j) \sub \NA(M)$. Then $W$  is an object  in $\ydRsmashrat$, and by Proposition \ref{prop:old} its Nichols algebra is isomorphic to $K_i^M$. This new information on $K_i^M$ is used to prove our second main result, Theorem \ref{theor:NtoN}, which says that
\begin{align}\label{intro2}
\Fu(K_i^M) \# \NA(M_i^*) \cong \NA(R_i(M)),
\end{align}
where the braided monidal functor $(\Fu,\fu)$ is defined with respect to the dual pair $(\NA(M_i^*),\NA(M_i))$. The left-hand side of \eqref{intro2} is the bosonization, hence a braided Hopf algebra in a natural way. In \cite[Thm.\,3.12(1)]{a-AHS10} a different algebra isomorphism
\begin{align}\label{intro3}
K_i^M \# \NA(M_i^*) \cong \NA(R_i(M)),
\end{align}
formally similar to \eqref{intro2}, was obtained. But there, the left-hand side is not a bosonization, and a priori it is only an algebra and not a braided Hopf algebra. This is the reason why the proof of \eqref{intro3} was quite involved. The Hopf algebra structure of $K_i^M \# \NA(M_i^*)$ induced from the isomorphism \eqref{intro3} was determined in \cite[Theorem 4.2]{p-HeckSchn09}.

If $R$ is an algebra and $M$ is a right $R$-module, we denote its module structure by $\mu^R_M = \mu_M : M \ot R \to M$. If $C$ is a coalgebra and $M$ is a right $C$-comodule, we denote by $\delta_M^C = \delta_M : M \to M \ot C$ the comodule structure map. The same notations $\mu^R_M$ and $\delta_M^C$ will be used for left modules and left comodules. In the following we assume that $H$ is a Hopf algebra over $\fie $ with comultiplication $\Delta = \Delta_H : H \to H \ot H, \; h \mapsto h\sw1 \ot h\sw2$, augmentation $\varepsilon = \varepsilon_H$, and
bijective antipode $\cS$.

\section{Preliminaries on bosonization of Yetter-Drinfeld Hopf algebras}
\label{sec:bosonization}


We recall some well-known notions and results (see e.\,g.~\cite[Sect.~1.4]{a-AHS10}), and note some useful formulas from the theory of Yetter-Drinfeld Hopf algebras.

A left Yetter-Drinfeld module over $H$ is a left $H$-module and a left $H$-comodule with $H$-action and $H$-coaction denoted by
$H \ot V \to V,\; h \ot v \mapsto h \lact v$, and
$\delta=\delta_V : V \to H \ot V,\; v \mapsto \delta(v) = v\_{-1} \ot v\_0$,
such that
\begin{align}
\delta(h \lact v) = h\_1 v\_{-1} \cS(h\_3) \ot h\_2 \lact v\_0\label{YD}
\end{align}
for all $v \in V,h \in H$.

The category of left Yetter-Drinfeld modules over $H$ with $H$-linear and $H$-colinear maps as morphisms is denoted by $\ydH$. It is a monoidal and braided category. If $V,W \in \ydH$, then the tensor product is the vector space $V \ot W$ with diagonal action and coaction given by
\begin{align}
h \lact (v \ot w) &= h\_1 \lact v \ot h\_2 \lact w,\\
\delta(v \ot w) &= v\_{-1} w\_{-1} \ot v\_{0} \ot w\_{0},\\
\intertext{and the braiding is defined by}
c_{V,W} : V &\ot W \to W \ot V,\; v \ot w \mapsto v\_{-1} \lact w \ot v\_0\label{braiding},
\intertext{with inverse}
c_{V,W}^{-1} : W &\ot V \to V \ot W,\; w \ot v \mapsto v\sw0 \ot \cS^{-1}(v\sw{-1}) \lact w,
\end{align}
for all $h \in H, v \in V,w \in W$.

The category $\rydH$ is defined in a similar way, where the objects are the right Yetter-Drinfeld modules over $H$, that is, right $H$ modules and right $H$-comodules $V$ such that
\begin{align}
\delta(v \lact h) = v\_0 \lact h\_2 \ot \cS(h\_1) v\_1 h\_3
\end{align}
for all $v \in V,h \in H$.
The monoidal structure is given by diagonal action and coaction, and the braiding is defined by
\begin{align}\label{braidingright}
c_{V,W} : V \ot W \to W \ot V,\; v \ot w \mapsto w\sw0 \ot v \lact w\sw1,
\end{align}
for all $V,W \in \rydH$.

We note that for any object $V \in \ydH$, there is a linear isomorphism
\begin{align}
&\theta_V : V \xrightarrow{\cong} V,\; v \mapsto \cS(v\sw{-1}) \lact v\sw0,\\
 \intertext{with inverse}
 &V \xrightarrow{\cong} V, \; v \mapsto \cS^{-2}(v\sw{-1}) \lact v\sw0.\label{ydiso}
\end{align}
The map $\theta_V$ is not a morphism in $\ydH$, but
\begin{align}
\theta _V(h\lact v) &= \cS^2(h) \lact \theta_V(v),\label{theta1}\\
\delta (\theta_V(v)) &=\cS^2(v\sw{-1}) \ot \theta_V(v\sw0)\label{theta2}
\end{align}
for all $v \in V, h \in H$, where $\delta(v) = v\sw{-1} \ot v\sw0$.

If $A,B$ are algebras in $\ydH$, then the algebra structure of the tensor product $A \ot B$ of the vector spaces $A,B$ is defined in terms of the braiding by
\begin{align}
(a \ot b)(a'\ot b') = a (b\sw{-1} \lact a') \ot b\sw0 b'
\end{align}
for all $a,a' \in A$ and $b,b'\in B$.

Let $R$ be a Hopf algebra in the braided monoidal category $\ydH $ with augmentation $\varepsilon_R : R \to \fie$, comultiplication $\copr _R:R\to R\ot R,\; r \mapsto
r\swo1 \ot r\swo2,$ and antipode $\cS_R$.  Thus $\varepsilon_R, \Delta_R, \cS_R$  are morphisms in $\ydH$ satisfying the Hopf algebra axioms.
The map $\cS_R$  anticommutes with multiplication and comultiplication in the following way.
\begin{align}
  \label{ruleS1}
  \cS_R(rs)=&\cS_R(r\_{-1}\lact s)\cS_R(r\_0),\\
  \copr _R(\cS_R(r))=&\cS_R(r\^1{}\_{-1}\lact r\^2)\ot \cS_R(r\^1{}\_0)
  \label{ruleS2}
 \end{align}
for all $r,s\in R$.

Let $A = R\#H$ be the bosonization of $R$.
As an algebra, $A$ is the smash product given by the $H$-action on $R$ with multiplication
\begin{align}
  \label{smashproduct}
  (r\#h)(r'\#h')=r(h\_1\lact r')\#h\_2h'
\end{align}
for all $r,r'\in R$, $h,h'\in H$. We will identify $r \# 1$ with $r$ and $1 \# h$ with $h$. Thus we view $R \sub A$ and $H \sub A$ as subalgebras, and the multiplication map
$$ R \ot H \to A, \;r \ot h \mapsto rh = r \# h,$$
 is bijective. Since $\lact$ denotes the $H$-action, we will always write $ab$ for the product of elements $a,b \in A$ (and not $a \lact b$). Note that
\begin{align}
hr &= (h\sw1 \lact r) h\sw2,\label{smash1}\\
rh &= h\sw2 (\cS^{-1}(h\sw1) \lact r)\label{smash2}
\end{align}
for all $r \in R, h \in H$.
As a coalgebra, $A$ is the cosmash product given by the $H$-coaction of the coalgebra $R$. We will denote its comultiplication by
$$\Delta : A \to A \ot A,\; a \mapsto a\sw1 \ot a\sw2.$$
 By definition,
\begin{align}
(rh)\sw1 \ot (rh)\sw2 = r\swo1 {r\swo2}\sw{-1} h\sw1 \ot {r\swo2}\sw0 h\sw2\label{cosmash}
\end{align}
for all $r \in R, h \in H$. Thus the projection maps
\begin{align}
&\pi : A \to H,\; r \# h \mapsto \varepsilon_R(r)h,\\
&\vartheta : A \to R, \; r \# h \mapsto r \varepsilon(h), \label{eq:defvartheta}
\end{align}
 are coalgebra maps, and
$$A \to R \ot H, \;a \mapsto \vartheta(a\sw1) \ot \pi(a\sw2),$$
is bijective.

Then $A= R \# H$ is a Hopf algebra with antipode $\cS = \cS_A$, where the restriction of $\cS$ to $H$ is the antipode of $H$, and
\begin{align}
\cS(r)&=\cS(r\_{-1})\cS_R(r\_0),\label{bigS}\\
\intertext{hence}
\cS^2(r) &= \cS_R^2(\theta_R(r))\label{bigS2}
\end{align}
for all $r\in R$.

The map $\pi$ is a Hopf algebra projection, and the subalgebra $R \sub A$ is a left coideal subalgebra, that is, $\Delta(R) \sub A \ot R$, which is stable under $\cS^2$.

The structure of the braided Hopf algebra $R$ can be expressed in terms of the Hopf algebra $R \# H$ and the projection $\pi$:
\begin{align}
R = A^{\co H} &= \{r \in A \mid r\sw1 \ot \pi(r\sw2)  = r \ot 1 \},\\
h \lact r &= h\sw1 r \cS(h\sw2),\label{action}\\
r\sw{-1} \ot r\sw0 &= \pi(r\sw1) \ot r\sw2,\label{coaction}\\
r\swo1 \ot r\swo2 &= r\sw1 \pi\cS(r\sw2) \ot r\sw3,\label{comult}\\
\cS_R(r)&=\pi(r\sw1)\cS(r\sw2)\label{antipode}\\
\intertext{for all $h\in H$, $r\in R$. We list some formulas related to the  projection $\vartheta$.}
\vartheta(a) &= a\sw1 \pi\cS(a\sw2),\label{vartheta}\\
a &=\vartheta(a\sw1) \pi(a\sw2), \label{decomposition}\\
r\swo1 \ot r\swo2 &= \vartheta(r\sw1) \ot r\sw2,\label{comult1}\\
\vartheta(a)\swo1 \ot \vartheta(a)\swo2 &= \vartheta(a\sw1) \ot \vartheta(a\sw2),\label{comultvartheta}\\
\vartheta(a)\sw{-1} \ot \vartheta(a)\sw0 &= \pi(a\sw1 \cS(a\sw3)) \ot \vartheta(a\sw2)\label{coactionvartheta}\\
\intertext{for all $r \in R, a \in A$.}
\intertext{By \eqref{action}, the inclusion $R \sub A$ is an $H$-linear algebra map, where the $H$-action on $A$ is the adjoint action. By \eqref{comultvartheta} and \eqref{coactionvartheta},
the map $\vartheta : A \to R$ is an $H$-colinear coalgebra map, where the $H$-coaction of $A$ is defined by}
A &\to H \ot A,\; a \mapsto \pi(a\sw1 S(a\sw3)) \ot a\sw2,
\intertext{that is, by the coadjoint $H$-coaction of $A$.}
\intertext{Finally we note the following useful formulas related to the behaviour of $\vartheta$ with respect to multiplication.}
\vartheta(ah) &= \varepsilon(h) \vartheta(a),\label{vartheta1}\\
\vartheta(ha) &= h \lact \vartheta(a),\label{vartheta2}
\end{align}
for all $h \in H,a \in A$.
\begin{lemma}\label{lem:vartheta3}
Let $R$ be a Hopf algebra in $\ydH$ and $A =R \# H$ its bosonization. Then
\begin{align*}
\vartheta\cS\left(a \pi\cS^{-1}(b\sw2) b\sw1\right) = \vartheta\cS(b\sw2) \left(\pi\left(\cS(b\sw1) b\sw3\right) \lact \vartheta\cS(a)\right),
\end{align*}
for all $h \in H$ and $a,b \in A$.
\end{lemma}
\begin{proof}
\begin{align*}
 \vartheta\cS(b\sw2) \left(\pi\left(\cS(b\sw1) b\sw3\right) \lact \vartheta\cS(a)\right)&=\vartheta\cS(b\sw3) \pi\left(\cS(b\sw2) b\sw4\right) \vartheta\cS(a)\pi\cS\left(\cS(b\sw1) b\sw5\right)\\
&=\cS(b\sw2) \pi(b\sw3) \cS(a\sw2) \pi\cS^2(a\sw1)\pi\cS\left(\cS(b\sw1) b\sw4\right)\\
&=\vartheta\left(\cS(b\sw1) \pi(b\sw2) \cS(a)\right)\\
&=\vartheta\cS\left(a \pi\cS^{-1}(b\sw2) b\sw1\right),
\end{align*}
where the second equality follows from \eqref{decomposition} applied to $\cS(b\sw2)$ and \eqref{vartheta} applied to $\cS(a)$, and the third equality follows from \eqref{vartheta}.
\end{proof}

It follows from \eqref{bigS2} and \eqref{ydiso} that the antipode $\cS_R$ of $R$ is bijective if and only if the antipode $\cS$ of $R$ is bijective. In this case the following formulas hold for $\cS_R^{-1}$ and $\cS^{-1}$.
\begin{align}
\cS_R^{-1}(r)&=\cS^{-1}(r\_0)r\_{-1} = \vartheta \cS^{-1}(r),\label{ruleS5}\\
\cS^{-1}(rh) &= \cS^{-1}(h) \cS_R^{-1}(r\sw0) \cS^{-1}(r\sw{-1})\label{bigSinvers}
\end{align}
for all $r,s \in R$.

\section{Dual pairs of braided Hopf algebras and rational modules}

The field $\fie$ will be considered as a topological space with the discrete topology.
We denote by $\cL_{\fie}$ the category of {\em linearly topologized vector spaces} over $\fie$. Objects of $\cL_{\fie}$ are topological vector spaces which have a basis of neighborhoods of $0$ consisting of vector subspaces. Morphisms in $\cL_{\fie}$ are continuous $\fie$-linear maps.

Thus an object in $\cL_{\fie}$ is a vector space and a topological space $V$, where the topology on $V$ is given by a  set $\{V_i \sub V \mid i \in I\}$ of vector subspaces of $V$ such that for all $i,j \in I$ there is an index $k \in I$ with $V_k \sub V_i \cap V_j$. The set $\{V_i \sub V \mid i \in I\}$ is a basis of neighborhoods of $0$, and a subset $U \sub V$ is open if and only if for all $x \in U$ there is an index $i \in I$ such that $x + V_i \sub U$.

In particular, a vector subspace $U \sub V$ is open if and only if $V_i \sub U$ for some $i \in I$.

Let $V, W \in \cL_{\fie}$, and let $\{V_i \sub V \mid i \in I\}$ and $\{W_j \sub W \mid j \in J\}$ be bases of neighborhoods of $0$. Then a linear map $f : V \to W$ is continuous if and only if for all $j \in J$ there is an index $i \in I$ with $f(V_i) \sub W_j$. We define the tensor product $V \ot W$ as an object in $\cL_{\fie}$ with
$$ \{V_i \ot W + V \ot W_j \mid (i,j) \in I \times J\}$$
as a basis of neighborhoods of $0$.

Let $R,R^{\vee}$ be vector spaces, and let
$$\langle\;,\;\rangle : R^{\vee}  \ot R \to \fie,\;\xi \ot x \mapsto \langle \xi,x \rangle,$$
be a $\fie$-bilinear form. If $X \sub R$ and $X' \sub R^{\vee}$ are subsets, we define
\begin{align*}
^{\perp}X &= \{ \xi \in R^{\vee} \mid \langle \xi,x\rangle = 0 \text { for all } x \in X\},\\
X'^{\perp} &= \{ x \in R \mid \langle \xi,x\rangle = 0 \text { for all } \xi \in X'\}.
 \end{align*}
We endow $R^{\vee}$ with the  {\em finite topology} ({or the \em weak topology}), which is the coarsest topology on $R^{\vee}$ such that the  evaluation maps $\langle\;,x\rangle : R^{\vee} \to \fie,\; \xi \mapsto \langle\xi,x\rangle,$ for all $x \in R$ are continuous.
In the same way we view $R$ as a topological space with the finite topology with respect to the evaluation maps $\langle\xi,\;\rangle : R \to \fie,\; x \mapsto \langle \xi,x\rangle$, for all $\xi \in R^{\vee}$.

Let $\cE$ be a cofinal subset of the set of all finite-dimensional subspaces of $R$ (that is, $\cE$ is a set of finite-dimensional subspaces of $R$, and any finite-dimensional subspace $E \sub R$ is contained in some $E_1 \in \cE$). Let $\cE'$ be a cofinal subset of the set of all finite-dimensional subspaces of $R^{\vee}$. Then $R^{\vee}$ and $R$ are objects in $\cL_{\fie}$, where
\begin{equation*}
\{^{\perp}E \mid E \in \cE\} \text{ and }
\{E'^{\perp} \mid E' \in \cE'\}
\end{equation*}
are bases of neighborhoods of $0$ of $R^{\vee}$ and $R$, respectively.

The pairing $\langle\;,\; \rangle$ is called {\em non-degenerate}
if $^{\perp}R =0 \text{ and } {R^{\vee}}^{\perp} =0$. Let $E \in \cE$, and assume that $^{\perp}R = 0$. Then
$$E \to (R^{\vee}/^{\perp}E)^*, \; x \mapsto (\overline{\xi} \mapsto \langle \xi,x\rangle),$$
 is injective. Since
 $$R^{\vee}/^{\perp}E \to E^*,\; \overline{\xi} \mapsto (x \mapsto \langle \xi,x \rangle),$$
  is injective by definition, it follows that
\begin{align}
R^{\vee}/^{\perp}E &\xrightarrow{\cong} E^*,\;\overline{\xi} \mapsto \langle \xi,\; \rangle,\label{nonde1}
\end{align}
is bijective.
By the same argument, for all $E' \in \cE_{R^{\vee}}$
\begin{align}
R/E'^{\perp} &\xrightarrow{\cong} {E'}^*,\;\overline{x} \mapsto \langle \;,x\rangle,\label{nonde2}
\end{align}
is bijective, if ${R^{\vee}}^{\perp} =0$.

If $V,W$ are vector spaces, denote by
$$\Hom_{\rat}(R^{\vee} \ot V,W) \text{ (respectively }\Hom_{\rat}(V \ot R^{\vee},W))$$ the set of all linear maps $g : R^{\vee} \ot V \to W$ (respectively $g : V \ot R^{\vee}  \to W$) such that for all $v \in V$   there is a finite-dimensional subspace $E \sub R$ with $g(^{\perp}E \ot v) = 0$ (respectively $g(v \ot {^{\perp}E}) = 0$).

\begin{lemma}\label{lem:rational}
Let $\langle \;,\; \rangle : R^{\vee} \ot R \to \fie$ be  a non-degenerate $\fie$-bilinear form of vector spaces, and
let $V,W$ be vector spaces. Then the following hold.
\begin{enumerate}
\item The map
$$ \dual : \Hom(V,R \ot W) \to \Hom_{\rat}(R^{\vee} \ot V,W),\;\; f \mapsto (\langle\;,\; \rangle \ot \id)(\id \ot f),$$
is bijective.
\item The map
$$ \dual' : \Hom(V,R \ot W) \to \Hom_{\rat}(V \ot R^{\vee},W),\;\; f \mapsto (\id \ot \langle\;,\; \rangle)\tau(f \ot \id),$$
is bijective, where $\tau : R \ot W \ot R^{\vee} \to W \ot R^{\vee} \ot R$ is the twist map with $\tau(x \ot w \ot \xi) = w \ot \xi \ot x$ for all $x \in R,w \in W, \xi \in R^{\vee}$.
\end{enumerate}
\end{lemma}
\begin{proof}
(1) For completeness we recall the following well-known argument.

Let $f \in \Hom(V,R \ot W)$, and $g=\dual(f)$. For all $v \in V$ there is a finite-dimensional subspace $E \sub R$ with $f(v) \in E \ot W$, hence $g(^{\perp}E \ot v) =0$. Thus $g \in \Hom_{\rat}(R^{\vee} \ot V,W)$.

Conversely, let $g \in \Hom_{\rat}(R^{\vee} \ot V,W)$. For any finite-dimensional subspace $U \sub V$ there is a finite-dimensional subspace $E \sub R$ with $g(^{\perp}E \ot U)=0$. Let $g_{U,E} \in \Hom(R^{\vee}/^{\perp}E \ot U, W)$ be the map induced by $g$, and $f_{U,E} \in \Hom(U,E \ot W)$ the inverse image of $g_{U,E}$ under the isomorphisms
\begin{align*}
\Hom(U,E \ot W) \xrightarrow{\cong} \Hom(E^* \ot U,W) \xrightarrow{\cong} \Hom(R^{\vee}/^{\perp}E \ot U,W),
\end{align*}
where the first map is the canonical isomorphism, and the second map is induced by the isomorphism in \eqref{nonde1}.

If $E'$ is a finite-dimensional subspace of $R$ containing $E$, then
$$f_{U,E}(v) = f_{U,E'}(v) \text{ for all } v \in U.$$
Hence $f_U \in \Hom(U,R \ot W)$, defined by $f_U(v) = f_{U,E}(v)$ for all $v \in U$, does not depend on the choice of $E$.

Since $f_{U'} \res U = f_U$ for all finite-dimensional subspaces $U \sub U'$ of $V$, the inverse image $\dual^{-1}(g)$ can be defined by the family $(f_U)$.

(2) follows from (1) since the twist map $V \ot R^{\vee} \to R^{\vee} \ot V$ defines an isomorphism $\Hom_{\rat}(R^{\vee} \ot V,W) \cong \Hom_{\rat}(V \ot R^{\vee},W)$.
\end{proof}

Let $R,R^{\vee}$ be Hopf algebras  in the braided monoidal category $\ydH$, and let
$$\langle\;,\;\rangle : R^{\vee}  \ot R \to \fie,\;\xi \ot x \mapsto \langle \xi,x \rangle,$$
be a $\fie$-bilinear form of vector spaces.

\begin{defin}\label{defin:pair}
Assume
that there are cofinal subsets $\cE_R$ (respectively $\cE_{R^{\vee}}$) of the sets of all finite-dimensional vector subspaces of $R$ (respectively of $R^{\vee}$) consisting of subobjects in $\ydH$.

Then the pair $(R,R^{\vee})$ together with the bilinear form $\langle \;,\; \rangle : R^{\vee} \ot R \to \fie$ is  called a {\em dual pair of Hopf algebras} in $\ydH$ if
\begin{align}
\langle\;,\;\rangle& \text{ is non-degenerate},\label{pair1}\\
\langle h \lact \xi,x\rangle &= \langle \xi,\cS(h) \lact x\rangle,\label{pair2}\\
\xi\sw{-1} \langle\xi\sw0,x\rangle &= \cS^{-1}(x\sw{-1}) \langle \xi,x\sw0\rangle,\label{pair3}\\
\langle \xi,xy\rangle &= \langle \xi\swo1,y\rangle \langle \xi\swo2,x\rangle,\;\langle1,x \rangle = \varepsilon(x),\label{pair4}\\
\langle \xi \eta, x\rangle &= \langle \xi, x\swo2\rangle \langle \eta,x\swo1\rangle,\;\langle\xi,1\rangle = \varepsilon(\xi),\label{pair5}\\
\Delta_{R^{\vee}} : R^{\vee} &\to R^{\vee} \ot R^{\vee} \text{ is continuous},\label{pair7}\\
\Delta_{R} : R &\to R \ot R \text{ is continuous}\label{pair8}
\end{align}
for all $x,y \in R, \xi,\eta \in R^{\vee}$ and $h \in H$.

A left or right $R^{\vee}$-module (respectively $R$-module) $M$ is called {\em rational} if any element of $M$ is annihilated by $^{\perp}E$ (respectively $E'^{\perp}$) for some finite-dimensional vector subspace $E \sub R$ (respectively $E'\sub R^{\vee}$).
\end{defin}

\begin{lemma}\label{lem:pair}
Let $(R,R^{\vee})$ together with $\langle \;,\; \rangle : R^{\vee} \ot R \to \fie$ be a dual pair of Hopf algebras in $\ydH$. Then for all $x \in R, \xi \in R^{\vee}$ and for all $E \in \cE_R, E' \in \cE_{R^{\vee}}$,
\begin{align}
&\langle \cS_{R^{\vee}}(\xi),x \rangle = \langle \xi,\cS_R(x) \rangle,\label{pair6}\\
&^{\perp}E \sub R^{\vee} \text{ and } E'^{\perp} \sub R \text{ are subobjects in }\ydH. \label{pair9}
\end{align}
\end{lemma}
\begin{proof}
The vector space $\Hom(R^{\vee}, R^{*\op})$ is an algebra with convolution product. We define linear maps $\varphi_1,\varphi_2, \psi \in \Hom(R^{\vee}, R^{*\op})$ by
$$\varphi_1(\xi)(x) = \langle \xi, \cS_R(x) \rangle,\; \varphi_2(\xi)(x) = \langle \cS_{R^{\vee}}(\xi), x \rangle,\; \psi(\xi)(x) = \langle \xi,x \rangle,$$
for all $\xi \in R^{\vee}, x \in R$. Then by \eqref{pair4} and \eqref{pair5} the unit element in $\Hom(R^{\vee}, R^{*\op})$ is equal to $\varphi_1 * \psi$ and also to $\psi * \varphi_2$. Hence $\varphi_1 = \varphi_2$.

\eqref{pair9} follows from \eqref{pair2} and \eqref{pair3}.
\end{proof}

Note that the bilinear form $\langle\;,\;\rangle : R^{\vee}  \ot R \to \fie$ is a morphism in $\ydH$ if and only if \eqref{pair2} and \eqref{pair3} are satisfied.

The continuity conditions \eqref{pair7} and \eqref{pair8} are equivalent to the following. For all  $E \in \cE_R$ and $E' \in \cE_{R^{\vee}}$ there  are $E_1 \in \cE_R$ and $E'_1 \in \cE_{R^{\vee}}$ such that
$$\Delta_{R^{\vee}}({}^{\perp}E_1) \sub {^{\perp}E} \ot R^{\vee} + R^{\vee} \ot {}^{\perp}E, \;
\Delta_{R}({E_1'^{\perp}}) \sub E'{^{\perp}} \ot R + R \ot E'{^{\perp}}.$$

By \eqref{nonde1} and \eqref{nonde2}, rational modules over $R$ or $R^{\vee}$ are locally finite. Recall that a module over an algebra is {\em locally finite} if each element of the module is contained in a finite-dimensional submodule.

\begin{examp}\label{exa:gradedpair}
Let $R^{\vee}= \oplus_{n \geq 0} R^{\vee}(n)$ and $R= \oplus_{n \geq 0} R(n)$ be $\no$-graded Hopf algebras in $\ydH$ with finite-dimensional components $R^{\vee}(n)$ and $R(n)$ for all $n \geq 0$, and let $\langle \;,\; \rangle : R^{\vee} \ot R \to \fie$ be a  bilinear form of vector spaces such that
\begin{align}
&\langle R^{\vee}(m), R(n) \rangle = 0 \text{  for all } n \neq m \text{  in }\no.\label{gradedpair}
\end{align}
Assume \eqref{pair1} -- \eqref{pair5}.

For all integers $n \geq 0$ we define
$$\cF_nR = \oplus_{i=0}^{n} R(i),\; \cF_nR^{\vee} = \oplus_{i=0}^{n} R^{\vee}(i).$$
Then the subspaces $\cF_n R \sub R, n \geq 0,$ and $\cF_n R^{\vee} \sub R^{\vee}, n\geq 0,$ form cofinal subsets of the set of all finite-dimensional subspaces of $R$ and of $R^{\vee}$ consisting of subobjects in $\ydH$.
For all $n \geq 0$, let
$$\cF^nR = \oplus_{i \geq n} R(i),\; \cF^nR^{\vee} = \oplus_{i \geq n} R^{\vee}(i).$$
Then by \eqref{gradedpair} and \eqref{pair1}, for all $n \geq 0$,
\begin{align}\label{gradedperp}
{}^{\perp}(\cF_{n-1}R) = \cF^n R^{\vee}, \;({\cF_{n-1}R^{\vee}})^{\perp} = \cF^nR.
\end{align}
 Since the coalgebras $R^{\vee}$ and $R$ are $\no$-graded, it follows that
$$\Delta_{R^{\vee}}(\cF^{2n} R^{\vee}) \sub  \cF^{n} R^{\vee} \otimes R^{\vee} + R^{\vee} \otimes \cF^{n} R^{\vee} ,\;\Delta_{R}(\cF^{2n} R) \sub \cF^n R \otimes R + R \otimes \cF^n R$$
for all $n \geq 0$. Thus $\Delta_R$ and $\Delta_{R^{\vee}}$ are continuous.

Hence the pair $(R,R^{\vee})$ together with the bilinear form $\langle \;,\; \rangle $ is  a dual pair of Hopf algebras in $\ydH$. Moreover, the remaining structure maps of $R^{\vee}$ and of $R$, that is multiplication, unit map, augmentation and antipode, are all continuous, since they are $\no$-graded. Here, the ground field is graded by $\fie(0) = \fie$, and $\fie(n) = 0$ for all $n \geq 1$.

Since $R(0)$ is a finite-dimensional Hopf algebra in $\ydH$, the antipode of $R(0)$ is bijective by \cite[Proposition 7.1]{inp-Takeuchi00}. Hence the Hopf subalgebra $\cF_0 R\# H$ of $R \# H$ has bijective antipode by \eqref{bigS2} and \eqref{ydiso}. The filtration
$$\cF_0R \# H \sub \cF_1R \# H \sub \cF_2R \# H \sub \cdots \sub R \# H$$
is a coalgebra filtration, and by the argument in \cite[Remark 2.1]{p-HeckSchn09}, the antipodes of $R \# H$ and of $R$ are bijective. The same proof shows that the antipodes of $R^{\vee} \# H$ and of $R^{\vee}$ are bijective.
\end{examp}

Let $(R,R^{\vee})$ together with $\langle \;,\; \rangle : R^{\vee} \ot R \to \fie$ be a dual pair of Hopf algebras in $\ydH$.
We denote by ${^R(\ydH)}$ the category of left $R$-comodules in the monoidal category $\ydH$, and by ${_{R^{\vee}}(\ydH)}_{\rat}$ the category of left $R^{\vee}$-modules in $\ydH$ which are rational as $R^{\vee}$-modules.

\begin{propo}\label{propo:rational}
\begin{enumerate}
\item For all $M \in {^R(\ydH)}$  let $D(M)=M$ as an object in $\ydH$ with $R^{\vee}$-module structure given by
\begin{align*}
\xi m = \langle \xi,m\sws{-1} \rangle m\sws0
\end{align*}
for all $\xi \in R^{\vee},m \in M$,where the left $R$-comodule structure of $M$ is denoted by $\delta_M(m)= m\sws{-1} \ot m\sws0$. Then $D(M) \in {_{R^{\vee}}(\ydH)}_{\rat}$.
\item The functor
\begin{align*}
D : {^R(\ydH)} \to {_{R^{\vee}}(\ydH)_{\rat}}
\end{align*}
mapping $M \in {^R(\ydH)}$ onto $D(M)$, and with $D(f) = f$ for all morphisms in ${^R(\ydH)}$, is an isomorphism of categories.
\end{enumerate}
\end{propo}
\begin{proof}
This follows from Lemma \ref{lem:rational} together with \eqref{pair2} -- \eqref{pair5}.
\end{proof}

\begin{lemma}\label{lem:rationaltensor}
The trivial left $R \# H$-module $\fie$ is rational as an $R$-module (by restriction). Let $V,W$ be  left $R \# H$-modules, and $V \ot W$ the left $R \#H$-module given by diagonal action. If $V$ and $W$ are rational as left $R$-modules, then $V \ot W$ is a rational $R$-module.
\end{lemma}
\begin{proof}
The trivial $R$-module $\fie$ is rational since for all $x \in (\fie 1_{R^{\vee}})^{\perp}$,
\begin{align*}
x 1_{\fie} = \varepsilon(x) = \langle 1_{R^{\vee}},x \rangle = 0
\end{align*}
by \eqref{pair4}.

To prove that $V \ot W$ is rational as an $R$-module, let $v \in V,w \in W$. It is enough to show that $E^{\perp} (v \ot w) =0$ for some $E \in \cE_{R^{\vee}}$. Since $V$ and $W$ are rational $R$-modules, there are $E_1, E_2 \in \cE_{R^{\vee}}$ with $E_1^{\perp} v =0, E_2^{\perp} w =0$. Let $E_3 \in \cE_{R^{\vee}}$ with $E_1 + E_2 \sub E_3$. Then $E_3^{\perp}v=0, E_3^{\perp}w=0$. By \eqref{pair8} there is a subspace $E \in \cE_{R^{\vee}}$ such that
\begin{align}
&\Delta_R(E^{\perp}) \sub E_3^{\perp} \ot R + R \ot E_3^{\perp}.\label{continuity}
\end{align}
Let $r \in E^{\perp}$. Then by \eqref{cosmash},
\begin{align}
r(v \ot w) = r\swo1 {r\swo2}\sw{-1} v \ot {r\swo2}\sw0 w.\label{action1}
\end{align}
We rewrite the first tensorand on the right-hand side in \eqref{action1} according to the multiplication rule \eqref{smash2} for elements in $R \# H$. Then the equality $r(v \ot w)=0$ follows from \eqref{continuity}, \eqref{action1} and \eqref{pair9}.
\end{proof}
Lemma \ref{lem:rationaltensor} also holds for $R^{\vee}$ instead of $R$ using \eqref{pair5} and \eqref{pair7} instead of \eqref{pair4} and \eqref{pair8}.

\begin{lemma}\label{lem:inversepair}
Assume that the antipodes of $R$ and of $R^{\vee}$ are bijective. Define  $\langle\;,\;\rangle' : R \ot R^{\vee} \to \fie$ by
\begin{align}
\langle x,\xi \rangle' = \langle \xi,\cS^2(x) \rangle\label{defin:inversepair}
\end{align}
for all $x \in R,\xi \in R^{\vee}$, where $\cS$ is the antipode of $R \# H$. Then $(R^{\vee},R)$ together with $\langle \;,\; \rangle' : R \ot R^{\vee} \to \fie$ is a dual pair of Hopf algebras in $\ydH$.
\end{lemma}
\begin{proof}
Using \eqref{bigS2}, \eqref{pair1} -- \eqref{pair5} for $\langle\;,\; \rangle'$ are easily checked.

We denote by $\perp$ (respectively $\perp'$) the complements with respect to $\langle\;,\; \rangle$ (respectively to $\langle \;,\; \rangle'$).

To prove \eqref{pair7} for $\langle\;,\; \rangle'$, we note that by \eqref{defin:inversepair} for all finite-dimensional subspaces $E \sub R$, $E^{\perp'} = {^{\perp}(\cS^2(E))}$. By assumption and \eqref{bigS2}, $\cS^2$ induces an isomorphism on $R$. Hence the weak topologies of $R^{\vee}$ defined with respect to $\langle\;,\; \rangle$ and to $\langle\;,\; \rangle'$ coincide, and \eqref{pair7} for $\langle\;,\; \rangle'$ follows.

To prove \eqref{pair8} for $\langle\;,\; \rangle'$, we again show that the weak topologies of $R$ defined with respect to $\langle\;,\; \rangle$ and to $\langle\;,\; \rangle'$ coincide. For all $x \in R, \xi \in R^{\vee}$,
\begin{align*}
\langle x, \xi \rangle' &= \langle \xi, \cS^2(x) \rangle && \\
&= \langle \xi, \cS_R^2(\cS(x\sw{-1}) \lact x\sw0) \rangle && \text{(by \eqref{bigS2})}\\
 &= \langle \cS_{R^{\vee}}^2(\xi), \cS(x\sw{-1}) \lact x\sw0 \rangle &&\text{(by \eqref{pair6})}.
 \end{align*}
Hence for all $E_1 \in \cE_{R^{\vee}}$,
\begin{align*}
{^{\perp'}E_1} &= \{ x \in R \mid \cS(x\sw{-1}) \lact x\sw0  \in (\cS_{R^{\vee}}^2(E_1))^{\perp}\}\\
&= (\cS_{R^{\vee}}^2(E_1))^{\perp},
\end{align*}
where the second equality follows from \eqref{ydiso} and \eqref{pair9}.
This proves our claim, since $\{\cS_{R^{\vee}}^2(E_1) \mid E_1 \in \cE_{R^{\vee}} \}$
is a cofinal subset of $\cE_{R^{\vee}}$ by the bijectivity of $\cS_{R^{\vee}}$.
\end{proof}

\section{Review of monoidal categories and their centers}\label{sec:categories}

Our reference for monoidal categories is \cite{b-Kassel1}, where the term tensor categories is used. Let $\cC$ and $\cD$ be strict monoidal categories, and $F : \cC \to \cD$ a  functor. We assume that $F(I)$ is the unit object in $\cD$. Let
$$\varphi = (\varphi_{X,Y} : F(X) \ot F(Y) \to F(X \ot Y))_{X,Y \in \cC}$$
 be a family of natural isomorphisms. Then $(F,\ph)$ is a {\em monoidal functor}  if for all $U,V,W \in \cC$
\begin{align}
\ph_{I,U} = \id_{F(U)} = \ph_{U,I},
\end{align}
and the diagram
\begin{align}\label{monoidal}
\begin{CD}
F(U) \ot F(V) \ot F(W) @>{\id \ot \varphi_{V,W}}>>  F(U) \ot F(V \ot W)\\
@V{\varphi_{U,V} \ot \id}VV    @V{\varphi_{U,V \ot W}}VV\\
F(U \ot V) \ot F(W) @>{\varphi_{U \ot V,W}}>>  F(U \ot V \ot W)
\end{CD}
\end{align}
commutes. A monoidal functor $(F,\ph)$ is called {\em strict} if $\ph = \id$. If $\cC$ and $\cD$ are strict braided monoidal categories, then a monoidal functor $(F,\ph)$ is {\em braided} if
for all $X,Y \in \cC$ the diagram
\begin{align}\label{braidedmonoidal}
\begin{CD}
F(X) \ot F(Y) @>{\varphi_{X,Y}}>>  F(X \ot Y)\\
@V{c_{F(X),F(Y)}}VV    @V{F(c_{X,Y})}VV\\
F(Y) \ot F(X) @>{\varphi_{Y,X}}>>  F(Y \ot X)
\end{CD}
\end{align}
commutes. A {\em monoidal equivalence} (respectively {\em isomorphism}) is a monoidal functor $(F,\ph)$ such that $F$ is an equivalence (respectively an isomorphism) of categories. Recall that a functor $F : \cC \to \cD$ is called an isomorphism if there is a functor $G : \cD \to \cC$ with $FG = \id_{\cD}$ and $GF = \id_{\cC}$.
A {\em braided monoidal equivalence} (respectively {\em isomorphism}) is a monoidal equivalence (respectively isomorphism) $(F,\varphi)$ such that $(F,\varphi)$ is a braided monoidal functor.

If $(F,\varphi) : \cC \to \cD$ and $(G ,\psi): \cD \to \cE$ are monoidal (respectively braided monoidal) functors, then the composition
\begin{align}\label{monoidalcomposition}
(GF,\lambda) : \cC \to \cE, \;\lambda_{X,Y} = G(\varphi_{X,Y})\psi_{F(X),F(Y)}, \text{ for all }X,Y \in \cC,
\end{align}
is a monoidal (respectively braided monoidal) functor.

Let $(F,\varphi) : \cC \to \cD$ be a monoidal isomorphism of categories with inverse functor $G : \cD \to \cC$. Then $(G,\psi)$ is a monoidal functor with
\begin{align}
\psi_{U,V}= G(\varphi_{G(U),G(V)})^{-1}: G(U) \ot G(V)  \to G(U \ot V)\label{inversemonoidal}
\end{align}
for all $U,V \in \cD$.

For later use we note the following  lemma.
\begin{lemma}\label{lem:diagram}
Let $\cC,\cD$ and $\cE$ be strict monoidal and braided categories, and $F : \cC \to \cD$ a functor. Let $(G,\psi) : \cD \to \cE$ and $(GF, \lambda) : \cC \to \cE$ be braided monoidal functors. Assume that  the functor $G$ is fully faithful.  Then there is exactly one family $\varphi = (\varphi_{X,Y})_{X,Y \in \cC}$ such that $(F,\varphi)$ is a braided monoidal functor and $$(GF,\lambda)=(\cC \xrightarrow{(F,\varphi)} \cD \xrightarrow{(G,\psi)} \cE).$$
\end{lemma}
\begin{proof}
Since $G$ is fully faithful, for all $X,Y \in \cC$ there is exactly one morphism $\varphi_{X,Y} : F(X) \ot F(Y) \to F(X \ot Y)$ with $\lambda_{X,Y} = G(\varphi_{X,Y})\psi_{F(X),F(Y)}$. Then one checks that $(F,\varphi)$ is a braided monoidal functor.
\end{proof}

We recall  the notion of the (left) {\em center} $\cZ(\cC)$ of a strict monoidal category $\cC$ with tensor product $\ot$ and unit object $I$ (see \cite[XIII.4]{b-Kassel1}, where the right center is discussed). Objects of $\cZ(\cC)$ are pairs $(M,\ga)$, where $M \in \cC$, and
$$\ga = (\ga_X : M \ot X \to X \ot M)_{X \in \cC}$$
 is a family of natural isomorphisms such that
\begin{align}
\ga_{X \ot Y} &= (\id_X \ot \ga_Y)(\ga_X \ot \id_Y)\label{objectcenter}
\end{align}
for all $X,Y \in \cC$.

Note that by \eqref{objectcenter}
\begin{align}
\ga_I = \id_M\label{objectcenterI}
\end{align}
for all $(M,\ga) \in \cZ(\cC)$.

A morphism  $f : (M,\ga) \to (N,\la)$ between objects  $(M,\ga)$ and $(N,\la)$ in $\cZ(\cC)$ is a morphism $f : M \to N$ in $\cC$ such that
\begin{align}\label{morphismcenter}
(\id_X \ot f) \ga_X = \la_X (f \ot \id_X)
\end{align}
for all $X \in \cC$. Composition of morphisms is given by the composition of morphisms in $\cC$. The category $\cZ(\cC)$ is braided monoidal. For objects $(M,\ga),(N,\la)$ in $\cZ(\cC)$ the tensor product is defined by
\begin{align}
(M,\ga) \ot (N,\la) &= (M \ot N, \si),\label{center1}\\
\si_X &= (\ga_X \ot \id_N)(\id_M \ot \la_X)\label{center2}
\end{align}
for all $X \in \cC$. The pair $(I, \id)$, where $\id_X = \id_{I \ot X}$ for all $X \in \cC$, is the unit in $\cZ(\cC)$. The braiding is defined by
\begin{align}\label{braidingcenter}
\ga_N : (M,\ga) \ot (N,\la) \to (N,\la) \ot (M,\ga).
\end{align}

We note that a monoidal isomorphism $(F,\varphi) : \cC \to \cD$ defines in the natural way a braided monoidal isomorphism between the centers of $\cC$ and $\cD$. For all objects $(M,\ga) \in \cC$ let
\begin{align}
F^{\cZ}(M,\ga) = (F(M), \ti{\ga}),
\end{align}
 and for all $X \in \cC$, the isomorphism $\ti{\ga}_{F(X)}$ is defined by the commutative diagram
\begin{align}
\begin{CD}
F(M) \ot F(X)  @>{\ti{\ga}_{F(X)}}>>  F(X) \ot F(M)\\
@V{\ph_{M,X}}VV    @V{\ph_{X,M}}VV\\
F(M \ot X) @>{F(\ga_X)}>>  F(X \ot M).
\end{CD}\label{defin:tilde}
\end{align}
For morphisms $f$ in $\cZ(\cC)$ we define $F^{\cZ}(f) = F(f)$. For $(M,\ga),(N,\la) \in \cZ(\cC)$ let
\begin{align}
{\ph}^{\cZ}_{(M,\ga),(N,\la)} = \ph_{M,N}.\label{monoidalcenter2}
\end{align}
Then the next lemma follows by carefully writing down the definitions.
\begin{lemma}\label{lem:center}
Let $(F,\varphi) : \cC \to \cD$ be a monoidal isomorphism. Then
\begin{align*}
(F^{\cZ},{\ph}^{\cZ})  : \cZ(\cC) \to \cZ(\cD)
\end{align*}
is a braided monoidal isomorphism.
\end{lemma}
Finally we note that we may view the categories of vector spaces and of modules or comodules over a Hopf algebra as strict monoidal categories since the associativity and unit constraints are given by functorial maps.

\section{Relative Yetter-Drinfeld modules}\label{sec:relative}

In this section we assume that $B,C$ are Hopf algebras with bijective antipode, $\rho: B \to C$ is a Hopf algebra homomorphism, and  $\cR \sub {_B\cM}$ is a full subcategory of the category of left $B$-modules closed under tensor products and containing the trivial left $B$-module $\fie$.

\begin{defin}\label{defin:ydBC}
We denote by $\ydBC_{\cR}$ the following monoidal category (depending on the map $\rho$). Objects of $\ydBC_{\cR}$ are left $B$-modules and left $C$-comodules $M$ with comodule structure $\delta : M \to C \ot M, m \mapsto m\sw{-1} \ot m\sw0,$ such that $M \in \cR$ as a module over $B$ and
\begin{align}\label{relativeYD}
\delta(bm) = \rho(b\sw1) m\sw{-1} \rho\cS(b\sw3) \ot b\sw2 m\sw0
\end{align}
for all $m \in M$ and $b \in B$. Morphisms are left $B$-linear and left $C$-colinear maps.

The tensor product $M \ot N$ of $M,N \in \ydBC_{\cR}$ is the tensor product of the vector spaces $M,N$ with diagonal action of $B$ and diagonal coaction of $C$.

We define $\ydBC = \ydBC_{\cR}$, when $\cR = {_B\cM}$ is the category of all $B$-modules. The full subcategory of $\ydB$ consisting of all objects $M \in \ydB$ with $M \in \cR$ as a $B$-module is denoted by $\ydB_{\cR}$.
\end{defin}

The Hopf algebra map $\rho : B \to C$ defines a functor
\begin{align}
{^{\rho}(\;)} : \ydB \to \ydBC,
\end{align}
mapping an object $M \in \ydB$  onto ${^{\rho}M}$, where ${^{\rho}M} = M$ as a $B$-module, and where ${^{\rho}M}$ is a $C$-comodule by $M \xrightarrow{\delta_M} B \ot M \xrightarrow{\rho \ot \id_M} C \ot M$.

Let
\begin{align}
\Phi : \ydB_{\cR} \to \cZ(\ydBC_{\cR})
\end{align}
be the functor defined on objects $M \in \ydB_{\cR}$ by
\begin{align}
\Phi(M) = (^{\rho}M,c_M),\;
c_{M,X} : M \ot X \to X \ot M,\; m \ot x \mapsto m\sw{-1} x \ot m \sw0,\label{defin:Phi}
\end{align}
for all $X \in \ydBC_{\cR}$, where $M \to B \ot M,\;m \mapsto m\sw{-1} \ot m\sw0,$ denotes the $B$-comodule structure of $M$. We let $\Phi(f) = f$ for morphisms $f$ in $\ydB_{\cR}$. It is easy to see that $\Phi$ is a well-defined functor.

We need the existence of enough objects in $\ydBC_{\cR}$.
\begin{defin}
The category $\ydBC_{\cR}$ is called {\em $B$-faithful} if the following conditions are satisfied.
\begin{align}
\text{For any }0 \neq b \in B, \;&bX \neq 0 \text{ for some }X \in \ydBC_{\cR}.\label{faithful1}\\
\text{For any }0 \neq t \in B \ot B, \;&t(X \ot Y) \neq 0 \text{ for some }X,Y \in \ydBC_{\cR}.\label{faithful2}
\end{align}
\end{defin}
\begin{examps}\label{exa:regularrep}
(1) Let $B$ be the left $B$-module with the regular representation, and the left $C$-comodule with the coadjoint coaction
\begin{align}
B \to C \ot B, \;b \mapsto \rho(b\sw1 \cS(b\sw3)) \ot b\sw2.\label{coadjoint}
 \end{align}
Then $B$ is an object in $\ydBC$. Since $bB\neq 0$, $t (B \ot B) \neq 0$ for all
$ 0 \neq b \in B$, $0 \neq t \in B \ot B$, the category $\ydBC$ is $B$-faithful.

(2) Let
\begin{align*}
R = \bigoplus_{n \in \no} R(n)
\end{align*}
 be an $\no$-graded Hopf algebra in $\ydH$, and $A = R \# H$ the bosonization. We define $\ydAH$ with respect to the Hopf algebra map $\pi : A \to H$. As in (1), $A$ with the regular representation and the coadjoint coaction with respect to $\pi$ defined in \eqref{coadjoint}, is an object in $\ydAH$.  The $H$-coaction $\delta_A : A \to H \ot A$ can be computed explicitly as
 \begin{align*}
 \delta_A(rh) = r\sw{-1} h\sw1 \cS(h\sw3) \ot r\sw0 h\sw2
 \end{align*}
 for all $r \in R,h \in H$.
 Hence it follows that for all $n \geq 0$,
 \begin{align*}
 \cF^nA = \bigoplus_{i \geq n} R(i) \ot H \sub A
\end{align*}
is an ideal and a left $H$-subcomodule of $A \in \ydAH$. Note that
\begin{align}
\bigcap_{n \geq 0} \cF^nA = 0,\;\bigcap_{n \geq 0} (\cF^nA \ot A + A \ot \cF^nA) = 0.
\end{align}
Hence for any $0 \neq a \in A, 0 \neq t \in A \ot A$ there is an integer $n \geq 0$ with
\begin{align*}
a(A/\cF^nA) \neq 0, \; t (A/\cF^nA \ot A/\cF^nA) \neq 0.
\end{align*}
Thus $\ydAH_{\cR}$ is $A$-faithful for all full subcategories $\cR$  of ${_A\cM}$  such that $A/\cF^nA \in \ydAH_{\cR}$ for all $n \geq 0$.
Note that for all $n \geq0$, $A/\cF^nA$ as an $R$-module is annihilated by $\oplus _{i \geq n} R(i)$.
\end{examps}

\begin{propo}\label{propo:modulecenter}
Assume that $\ydBC_{\cR}$ is $B$-faithful.
\begin{enumerate}
\item The functor $\Phi : \ydB_{\cR} \to \cZ(\ydBC_{\cR})$ is fully faithful, strict monoidal and braided.\label{modulecenter1}
\item Let $(M,\ga) \in \cZ(\ydBC_{\cR})$ with comodule structure $\delta_M : M \to C \ot M$. Assume that there is a $\fie$-linear map $\widetilde{\delta}_M : M \to B \ot M$, denoted by $\widetilde{\delta}_M(m) = m\swe{-1} \ot m\swe0$ for all $m \in M$, with
\begin{align}
    \ga_X (m \ot x) &= m\swe{-1} x \ot m\swe0\label{module1}\\
    \delta_M &= (\rho \ot \id_M)\widetilde{\delta}_M,\label{module2}
    \end{align}
for all $X \in \ydBC_{\cR}, x \in X$ and $m \in M$. Then the map $\widetilde{\delta}_M$ is uniquely determined. Let $\ti{M} =M$ as a $B$-module. Then  $\ti{M} \in \ydB_{\cR}$ with $B$-comodule structure $\widetilde{\delta}_M$, and $\Phi(\ti{M}) = (M,\ga)$.\label{modulecenter2}
\end{enumerate}
\end{propo}
\begin{proof}
(1) It is clear from the definitions that $\Phi$ is strict monoidal and braided, see \eqref{braidingright}, \eqref{braidingcenter} and \eqref{defin:Phi}. To prove that $\Phi$ is fully faithful, let $M,N \in \ydB$, and $f : \Phi(M) \to \Phi(N)$ a morphism in $\cZ(\ydBC_{\cR})$. In particular, $f : M \to N$ is a left $B$-linear and left $C$-colinear homomorphism.  We have to show that $f$ is left $B$-colinear.
Let $X \in \ydBC_{\cR}$, $m \in M$ and $x \in X$. Then
\begin{align}\label{equationfaithful}
f(m)\sw{-1}x \ot f(m)\sw0 = m\sw{-1}x \ot f(m\sw0),
\end{align}
since $f$ is a morphism in $\cZ(\ydBC_{\cR})$. It follows from \eqref{equationfaithful} and \eqref{faithful1} that
\begin{align*}
f(m)\sw{-1} \ot f(m)\sw0 = m\sw{-1} \ot f(m\sw0)
\end{align*}
 in $B \ot M$ for all $m \in M$, that is, $f$ is $B$-colinear.

(2) The map $\widetilde{\delta}_M$ is uniquely determined by \eqref{faithful1} and  \eqref{module1}. We have to show that $\ti{M}$ is a $B$-comodule with structure map $\widetilde{\delta}_M$, and that $\ti{M} \in \ydB_{\cR}$ with comodule structure $\widetilde{\delta}_M$ and the given $B$-module structure.

Let $X,Y \in \ydBC_{\cR}, x \in X,y \in Y$ and $m \in M$. By \eqref{objectcenter},
\begin{align*}
\Delta(m\swe{-1})(x \ot y) \ot m\swe0 = m\swe{-1}x \ot m \swe0{\swe{-1}}y \ot {m\swe0}{\swe0}.
\end{align*}
Hence $\widetilde{\delta}_M$ is coassociative by \eqref{faithful2}. Let $\fie \in \ydBC{_\cR}$ be the trivial object. Then by \eqref{objectcenterI},
\begin{align*}
1 \ot m = \ga_{\fie}(m \ot 1) = m\swe{-1} 1 \ot m \swe0 = 1 \ot \varepsilon(m\swe{-1}) m\sw0
\end{align*}
for all $m \in M$. Hence the  comultiplication $\widetilde{\delta}_M$ is counitary.

For all $X \in \ydBC_{\cR}$, the map $\ga_X$ is $B$-linear. Hence
\begin{align*}
(b\sw1m)\swe{-1} b\sw2x \ot (b\sw1m)\swe0 = b\sw1 m\swe{-1}x \ot b\sw2m \swe0
\end{align*}
for all $b \in B, m \in M$ and $x \in X$. Hence $\ti{M} \in \ydB_{\cR}$ by \eqref{faithful1}.

Finally $\Phi(\ti{M}) = (M,\ga)$ by \eqref{module1} and \eqref{module2}.
\end{proof}
\begin{remar}
In general, $\Phi : \ydB_{\cR} \to \cZ(\ydBC_{\cR})$ is not an equivalence.
However, in the case when $C = \fie$ and $\rho = \varepsilon$, hence $\ydBC = {_B\cM}$,
it is well-known (compare \cite{b-Kassel1} XIII.5) that $\Phi : \ydB \to \cZ({_B\cM})$
is an equivalence. Indeed, let $(M,\gamma) \in \cZ({_B\cM})$. Define
$m\swe{-1} \ot m\swe0 = \gamma_B(m \ot 1)$ for all $m \in M$,
where the $B$-module structure of $B \in {_B\cM}$ is given by multiplication.
Then for any $X \in {_B\cM}$ and $x \in X$ there is a $B$-linear map
$f : B \to X$ with $f(1) =x$, and $\gamma_X(m \ot x) = m\swe{-1} x \ot m\swe0$
by the naturality of $\gamma$. This proves \eqref{module1}.
Similarly, \eqref{module2} follows by considering the trivial $B$-module $\fie$ and the $B$-linear map $\varepsilon$.
 Moreover, ${_B\cM}$ is $B$-faithful by Example \ref{exa:regularrep} (1).
 Thus in this case the assumption in Proposition \ref{propo:modulecenter} \eqref{modulecenter2} is always satisfied.
\end{remar}
\begin{defin}\label{defin:rydBC}
We denote by $\rydBC$ the monoidal category whose objects are right $B$-modules and right $C$-comodules $M$ with comodule structure denoted by $\delta: M \to M \ot C,\; m \mapsto m\sw{0} \ot m\sw{-1}$, such that
\begin{align}
\delta(mb) = m\sw0  b\sw2 \ot \cS(\rho(b\sw1)) m\sw{1} \rho(b\sw3)\label{rydBC}
\end{align}
for all $m \in M$ and $b \in B$. Morphisms are right $B$-linear and right $C$-colinear maps.

The tensor product $M \ot N$ of $M,N \in \rydBC$ is the tensor product of the vector spaces $M,N$ with diagonal action of $B$ and diagonal coaction of $C$. The monoidal category  $\rydC$ is braided by \eqref{braidingright}.
\end{defin}
We define a functor
\begin{align}
\Psi : \rydC \to \cZ(\rydBC)
\end{align}
on objects $M \in \rydC$ by
\begin{align}
\Psi(M) = (M_{\rho},c_M),\;
c_{M,X} : M \ot X \to X \ot M,\; m \ot x \mapsto x\sw0 \ot m x\sw1,\label{defin:Psi}
\end{align}
for all $X \in \rydBC$, where $M_{\rho}$ is $M$ as a $B$-module  via $\rho$. We let $\Psi(f) = f$ for morphisms $f$ in $\rydC$.

\begin{examp}\label{exa:regularcorep}
Let $C$ be the regular corepresentation with right $C$-comodule structure given by the comultiplication $\Delta_C$ of $C$. We define a right $B$-module structure on $C$ by the adjoint action, that is
\begin{align}
c \vartriangleleft b = \rho\cS(b\sw1)c \rho(b\sw2)
\end{align}
for all $c \in C,b \in B$. Then $C$ is an object in $\rydBC$.
\end{examp}

\begin{propo}\label{propo:comodulecenter}
\begin{enumerate}
\item The functor $\Psi : \rydC \to \cZ(\rydBC)$
is fully faithful, strict monoidal and braided. \label{comodulecenter1}
\item Let $(M,\ga) \in \cZ(\rydBC)$ with module structure $\mu_M : M \ot B \to M$. Assume that there is a $\fie$-linear map $\widetilde{\mu}_M : M \ot C \to M$ such that
\begin{align}
    \ga_X (m \ot x) &= x\sw0 \ot \widetilde{\mu}_M(m \ot x\sw1),\label{comodule1}\\
    \mu_M &=\widetilde{\mu_M} (\id \ot \rho)\label{comodule2}
    \end{align}
for all $X \in \rydBC, x \in X$ and $m \in M$. Then the map $\widetilde{\mu}_M$ is uniquely determined.
Let $\ti{M} = M$ as a $C$-comodule.
 Then $\ti{M} \in \rydC$ with $C$-module structure $\widetilde{\mu}_M$, and
$\Psi(\ti{M}) = (M,\ga)$.
\label{comodulecenter2}
\end{enumerate}
\end{propo}
\begin{proof}
(1) Again it is clear that $\Psi$ is strict monoidal and braided. To see that $\Psi$ is fully faithful, let $M,N \in \rydC$ and $f : \Psi(M) \to \Psi(N)$ a morphism in $\cZ(\rydBC)$. We have to show that $f$ is right $C$-linear. Let $X =C \in \rydBC$ in Example \ref{exa:regularcorep}. Since $f$ is a morphism in $\cZ(\rydBC)$,
$$x\sw1 \ot f(mx\sw2) = x\sw1 \ot f(m)x\sw2$$
for all $x \in C, m \in M$. By applying $\varepsilon \ot \id$ to this equation it follows that $f$ is right $C$-linear.

(2) Let $C \in \rydBC$ as in Example \ref{exa:regularcorep}. Then $(\varepsilon \ot \id)\ga_C = \widetilde{\mu}_M$. Hence $\widetilde{\mu}_M$ is uniquely determined. Let $X=Y = C \in \rydBC$. By \eqref{objectcenter}
$$x\sw1 \ot y\sw1 \ot \widetilde{\mu}_M(m \ot x\sw2y\sw2) = x\sw1 \ot y\sw1 \ot \widetilde{\mu}_M(\widetilde{\mu}_M(m \ot x\sw2) \ot y\sw2)$$
for all $x,y \in C, m \in M$. By applying $\varepsilon \ot \varepsilon \ot \id$ it follows that $\widetilde{\mu}_M$ is associative. By \eqref{objectcenterI}, $\widetilde{\mu}_M$ is unitary.
We will write $mc = \widetilde{\mu}_M(m \ot c)$ for all $m \in M,c \in C$.

Since $\ga_C$ is right $C$-colinear,
$$x\sw1 \ot (m x\sw3)\sw0 \ot x\sw2 (m x\sw3)\sw1 = x\sw1 \ot m\sw0 x\sw2 \ot m\sw1 x\sw3$$
for all $x \in C, m \in M$. By applying $\varepsilon \ot \id$ it follows that $\ti{M} \in \rydC$.

Finally $\Psi(\ti{M}) = (M,\ga)$ by \eqref{comodule1} and \eqref{comodule2}.
\end{proof}

We fix an odd integer $l$, and assume that the antipodes of $B$ and $C$ are bijective.

Let $M \in \rydBC$  with $C$-comodule structure $\delta_M : M \to M \ot C, \;m \mapsto m\sw{0} \ot m\sw1$. We define an object $S_l(M) \in \ydBC$ by    $S_l(M) = M$ as a vector space with left
$B$-action and left $C$-coaction given by
\begin{align}
bm &= m \cS^{-l}(b),\label{leftright1}\\
\delta_{S_l(M)}(m) &= \cS^l(m\sw1) \ot m\sw0\label{leftright2}
\end{align}
for all $b \in B,m\in M$. For morphisms $f$ in $\rydBC$ we set $S_l(f) = f$.

Let $M \in \ydBC$ with comodule structure $\delta_M : M \to C \ot M, \; m \mapsto m\sw{-1} \ot m\sw0$. We define $S_l^{-1}(M) = M$ as a vector space with right $B$-action and right $C$-coaction given by
\begin{align}
mb &= \cS^l(b)m,\\
\delta_{S_l^{-1}(M)} (m) &= m \sw0 \ot \cS^{-l}(m\sw{-1})
\end{align}
for all $b \in B,m \in M$. For morphisms $f$ in $\ydBC$ we set $S_l^{-1}(f) = f$.

\begin{lemma}\label{lem:leftright}
Let $l$ be an odd integer, and assume that the antipodes of $B$ and $C$ are bijective.
\begin{enumerate}
\item The functor $S_l : \rydBC \to \ydBC$  mapping an object $M \in \rydBC$ onto $S_l(M)$, and a morphism $f$ onto $f$, is an isomorphism of categories with inverse $S_l^{-1}$.
\item Let $B=C=H$, and $\rho =\id_H$. Then $(S_l,\varphi) : \rydH \to \ydH$ is a braided monoidal isomorphism, where $\varphi$ is defined by
\begin{align*}
\varphi_{M,N} : S_l(M) \ot S_l(N) &\to S_l(M \ot N),\\
 m \ot n &\mapsto m \cS^{-1}(n\sw1) \ot n\sw0
= \cS^{-1}(n\sw{-1})m \ot n\sw0,
\end{align*}
for all $M,N \in \rydH$.

The inverse braided monoidal isomorphism is $(S_l^{-1},\psi) : \ydH \to \rydH$, where $\psi$ is defined by
\begin{align*}
\psi_{M,N} : S_l^{-1}(M) \ot S_l^{-1}(N) &\to S_l^{-1}(M \ot N),\\
m \ot n &\mapsto n\sw{-1} m \ot n\sw0
= m n\sw1 \ot n\sw0,
\end{align*}
for all $M,N \in \ydH$.
\end{enumerate}
\end{lemma}
\begin{proof}
(1) Let $M \in \rydBC$. Then $S_l(M) \in \ydBC$ since for all $m \in M,b \in B$,
\begin{align*}
\delta_{S_l(M)}(bm) = \delta_{S_l(M)}(m \cS^{-l}(b)) &= \cS^l\left(\rho\cS\cS^{-l}(b\sw3) m\sw1 \rho\cS^{-l}(b\sw1)\right) \ot m\sw0\cS^{-l}(b\sw2)\\
&=\rho(b\sw1) \cS^l(m\sw1) \cS\rho(b\sw3) \ot b\sw2 m\sw0.
\end{align*}
Thus $S_l$ is a well-defined functor. Similarly it follows that $S_l^{-1}$ is a well-defined functor.

(2) is shown in \cite[Proposition 2.2.1, 1.]{a-AndrGr99} for $l= -1$.
\end{proof}

\begin{remar}
 In general, it is not clear whether the functor $S_l$ in Lemma \ref{lem:leftright} is monoidal.
 This is one of the reasons why in the proof of our braided monoidal isomorphism of left Yetter-Drinfeld modules given in Theorem~\ref{theor:third}
 we have to change sides starting  in Theorem \ref{theor:first} with a monoidal isomorphism between relative right and left Yetter-Drinfeld modules.
\end{remar}

\section{The first isomorphism}\label{sec:first}

\begin{defin}\label{defin:relative}
Let $R$ be a Hopf algebra in $\ydH$.

We denote by $\ydRH$ and $\rydRH$ the categories $\ydRH$ and $\rydRH$ in Definition \ref{defin:ydBC} and \ref{defin:rydBC} with respect to the inclusion $H \sub R \#H$ as the Hopf algebra map $\rho$.

We denote by $\uydRH$ the category $\uydRH$ in Definition \ref{defin:ydBC} where $\rho$ is the Hopf algebra projection $\pi : R\#H \to H$ of $R \#H$.

Assume that $(R, R^{\vee})$ together with $\langle\;,\; \rangle$ is a dual pair of Hopf algebras in $\ydH$ with bijective antipodes. Then the antipodes of $R \# H$ and of $R^{\vee} \# H$ are bijective by \eqref{bigS2} and \eqref{ydiso}. We denote by $\ydRveeHrat$ (respectively $\ydRveesmashrat$) the full subcategory of objects of $\ydRveeH$ (respectively of $\ydRveesmash$)  which are rational as $R^{\vee}$-modules by restriction. The full subcategories of $\ydRsmash$ (respectively of $\rydRsmash$) consisting of objects which are rational over $R$ will be denoted by $\ydRsmashrat$ (respectively $\rydRsmashrat$).
\end{defin}

\begin{lemma}\label{lem:Rmod}
Let $R$ be a Hopf algebra in $\ydH$, and let $_R(\ydH)$ be the category of left $R$-modules in the monoidal category $\ydH$.
\begin{enumerate}
\item Let $M \in \uydRH$. Define $V_1(M) = M$ as a vector space and as a left $H$- and a left $R$-module by restriction of the $R\#H$-module structure. Then $V_1(M) \in \ydH$ with the given $H$-comodule structure, and the multiplication map $ R \ot M \to M$ is a morphism in $\ydH$.
\item The functor
\begin{align*}
V_1 : \uydRH \to {_R(\ydH)}
\end{align*}
 mapping objects $M \in \ydH$ to $V_1(M)$ and morphisms $f$ to $f$, is an isomorphism of categories.  The inverse functor $V_1^{-1}$ maps an object $M \in {_R(\ydH)}$ onto the vector space $M$ with given $H$-comodule structure and $R\#H$-module structure $R \#H \ot M \xrightarrow{\id_R \ot \mu^H_M} R \ot M \xrightarrow{\mu^R_M} M$.
\end{enumerate}
\end{lemma}
\begin{proof}
It follows from the definition of the smash product that $M$ is a left $R \#H$-module if and only if $\mu^R_M$ is $H$-linear.

The set of all elements $a \in R \#H$ satisfying the following Yetter-Drinfeld condition
 \begin{align}
\delta^H(am) = \pi(a\sw1) m\sw{-1} \pi\cS(a\sw3) \ot a\sw2 m \sw0\label{piyd}
\end{align}
for all $m \in M$ and $a \in R\#H$, is a subalgebra of $R \#H$. Hence \eqref{piyd} holds for all $a \in R\#H$ and $m \in M$ if and only if \eqref{piyd} holds for all $m \in M$ and $a \in R\cup H$. Note that \eqref{piyd} for all $m \in M$ and $a \in H$ is the Yetter-Drinfeld condition of $\ydH$, and \eqref{piyd} for all $m \in M$ and $a \in R$ says that $\mu^R_M$ is $H$-colinear, since for all $a \in R$, $a\sw1 \ot a\sw2 \ot a\sw3 \in R\#H \ot R\#H \ot R$, hence
\begin{align*}
a\sw1 \ot a\sw2 \ot \pi\cS(a\sw3) = a\sw1 \ot a\sw2 \ot 1.
\end{align*}
This proves the Lemma.
\end{proof}
\begin{lemma}\label{lem:Rcomod}
Let $R$ be a Hopf algebra in $\ydH$, and let $^R(\ydH)$ be the category of left $R$-comodules in the monoidal category $\ydH$.
\begin{enumerate}
\item Let $M \in \ydRH$ with comodule structure $\delta_M : M \to R \#H \ot M$. Define $V_2(M) = M$ as a vector space with left $H$-comodule structure $\delta_M^H$ and left $R$-comodule structure $\delta_M^R$ given by
    \begin{align*}
    \delta_M^H = (\pi \ot \id_M)\delta_M,\;\delta_M^R = (\vartheta \ot \id_M)\delta_M.
     \end{align*}
     Then $V_2(M) \in \ydH$ with $H$-comodule structure $\delta_M^H$ and the given $H$-module structure, and $\delta_M^R : M \to R \ot M$ is a morphism in $\ydH$.
\item The functor
\begin{align*}
V_2 : \ydRH \to {^R(\ydH)}
\end{align*}
 mapping objects $M \in \ydH$ to $V_2(M)$ and morphisms $f$ to $f$, is an isomorphism of categories. The inverse functor $V_2^{-1}$ maps an object $M \in {^R(\ydH)}$ onto the vector space $M$ with given $H$-module structure and $R \#H$-comodule structure $M \xrightarrow{\delta^R_M}  R \ot M \xrightarrow{\id_R \ot \delta^H_M} R \# H \ot M$.
\end{enumerate}
\end{lemma}
\begin{proof}
This is shown similarly to the proof of Lemma \ref{lem:Rmod}.
\end{proof}

For later use we note a formula for the right $R \#H$-comodule structure of a left $R \#H$-comodule defined via $\cS^{-1}$.

\begin{lemma}\label{lem:reformulation}
Let $R$ be a Hopf algebra in $\ydH$ with bijective antipode, $M$ a left $H$-comodule with $H$-coaction
$\delta^H : M \to H \ot M$, $m \mapsto m\sw{-1} \ot m\sw0$,
and
\begin{align*}
\delta^R : M &\to R \ot M,\; m \mapsto m\sws{-1} \ot m\sws0
\end{align*}
a linear map. Define $\delta : M \to R \#H \ot M, \; m \mapsto m\swe{-1} \ot m\swe0$,
by $\delta = (\id \ot \delta^H)\delta^R$. Then
\begin{align}
\vartheta\cS^{-1}(m\swe{-1}) \ot m\swe0 = \cS_R^{-1}\left(\cS^{-1}({m\sws0}\sw{-1}) \lact m\sws{-1}\right) \ot {m\sws0}\sw0
\end{align}
for all $m \in M$.
\end{lemma}
\begin{proof}
Let $m \in M$. Then $m\swe{-1} \ot m\swe0 = m\sws{-1} {m\sws0}\sw{-1} \ot {m\sws0}\sw0$, and
\begin{align*}
\cS^{-1}({m\sws0}\sw{-1}) \lact m\sws{-1} \ot {m\sws0}\sw0&= \cS^{-1}({m\sws0}\sw{-1}) m\sws{-1} {m\sws0}\sw{-2} \ot  {m\sws0}\sw0\\
&= \cS^{-1}({m\swe0}\sw{-1}) m\swe{-1} \ot {m\swe0}\sw0.
\end{align*}
Hence
\begin{align*}
&\cS_R^{-1}\left(\cS^{-1}({m\sws0}\sw{-1}) \lact m\sws{-1}\right) \ot {m\sws0}\sw0 &&\\
&\phantom{aa.}=\cS_R^{-1}\left(\cS^{-1}({m\swe0}\sw{-1})  m\swe{-1}\right) \ot {m\swe0}\sw0&&\\
&\phantom{aa.}= \vartheta \cS^{-1}\left(\cS^{-1}({m\swe0}\sw{-1})  m\swe{-1}\right) \ot {m\swe0}\sw0 &\text{ (by } \eqref{ruleS5})&\\
&\phantom{aa.}= \vartheta \cS^{-1}(m\swe{-1}) \ot {m\swe0}. &\text{ (by } \eqref{vartheta1})&
\end{align*}
\end{proof}

\begin{theor}\label{theor:first}
Let $(R, R^{\vee})$ be a dual pair of Hopf algebras in $\ydH$ with bijective antipodes and with bilinear form $\langle\;,\; \rangle$.

A monoidal isomorphism
$$(F,\varphi) : \rydRH \to \ydRveeHrat$$
is defined as follows.

For any object $M \in \rydRH$ with right $R \# H$-comodule structure denoted by
$$\delta_M: M \to M \ot R \# H,\; m \mapsto m\swe0 \ot m\swe1,$$
let $F(M)=M$ as a vector space and $F(M) \in \ydRveeH$
with left $H$-action, $H$-coaction $\delta_{F(M)}^H$ and $R^{\vee}$-action, respectively, given by
\begin{align}
hm &= m \cS^{-1}(h),\label{first1}\\
\delta_{F(M)}^H(m) &= \pi\cS(m\swe1) \ot m\swe0,\label{first2}\\
\xi m &= \langle \xi,\vartheta\cS(m\swe1)\rangle m\swe0\label{first3}
\end{align}
for all $h \in H,m \in M, \xi \in R^{\vee}$. For any morphism $f$ in $\rydRH$ let $F(f)=f$.
The natural transformation $\varphi$ is defined by
\begin{align}
\varphi_{M,N} : F(M) \ot F(N) \to F(M \ot N),\; m \ot n \mapsto m \pi\cS^{-1}(n\swe1) \ot n\swe0,
\end{align}
for all $M,N \in \rydRH$.
\end{theor}
\begin{proof}
The functor $F$ is the composition of the isomorphisms
\begin{align*}
\rydRH \xrightarrow{S} \ydRH \xrightarrow{V_2} {^R(\ydH)} \xrightarrow{D} {_{R^{\vee}}(\ydH)_{\rat}} \xrightarrow{V_1^{-1}} \ydRveeHrat,
\end{align*}
where $S=S_1$ is the isomorphism of Lemma \ref{lem:leftright}, $V_2$ is the isomorphism of Lemma \ref{lem:Rcomod}, $D$ is the isomorphism of Proposition \ref{propo:rational}, and where the last isomorphism is the restriction of $V_1^{-1}$ for $R^{\vee}$ of Lemma \ref{lem:Rmod} to rational objects.

 Let $M,N \in \rydRH$. The map
\begin{align*}
\varphi  = \varphi_{M,N} : F(M) \ot F(N) \to F(M \ot N)
\end{align*}
is a linear isomorphism  with $\varphi^{-1}(m \ot n)= m\pi(n\swe1) \ot n \swe0$ for all $m \in M,n\in N$. It follows from the Yetter-Drinfeld condition \eqref{rydBC} that $\varphi$ is an $H$-linear and $H$-colinear map, since for all $m \in M,n\in N$ and $h \in H$,
\begin{align*}
\varphi(h(m \ot n)) &= \varphi(m\cS^{-1}(h\sw1) \ot n \cS^{-1}(h\sw2))\\
&= m \cS^{-1}(h\sw1) \pi\cS^{-1}(h\sw4 n\swe1 \cS^{-1}(h\sw2)) \ot n\swe0 \cS^{-1}(h\sw3)\\
&=m\cS^{-1}(h\sw1) \cS^{-2}(h\sw2) \pi\cS^{-1}(n\swe1) \cS^{-1}(h\sw4) \ot n\swe0 \cS^{-1}(h\sw3)\\
&= h \varphi(m \ot n),
\end{align*}
\begin{align*}
\delta _{F(M\ot N)}^H\varphi(m \ot n) &= \pi\cS(\pi(n\swe4) m\swe1 \pi\cS^{-1}(n\swe2) n\swe1) \ot m\swe0 \pi\cS^{-1}(n\swe3) \ot n\swe0\\
&= \pi\cS(n\swe2 m\swe1) \ot m\swe0 \pi\cS^{-1}(n\swe1) \ot n\swe0\\
&=(\id_H \ot \varphi)\delta_{F(M) \ot F(N)}^H(m \ot n).
\end{align*}
To prove that $\varphi$
 is a left $R^{\vee}$-linear map, let $\xi \in R^{\vee}, m \in M$ and $n \in N$.
We first show that
\begin{align}\label{NR}
\xi\sw{-2} \ot \xi\sw{-1} \langle \xi\sw0,\vartheta\cS(a)\rangle = \pi( \cS(a\sw2) a\sw4) \ot \pi(\cS(a\sw1) a\sw5) \langle\xi,\vartheta\cS(a\sw3)\rangle
\end{align}
for all $a \in R \# H$.

By \eqref{coactionvartheta},
\begin{align*}
(\vartheta\cS(a))\sw{-2} \ot (\vartheta\cS(a))\sw{-1} &\ot (\vartheta\cS(a))\sw{0}\\
 &= \Delta(\pi(S(a\sw3) \cS^2(a\sw1)) \ot \vartheta\cS(a\sw2)\\
&=\pi(\cS(a\sw5) \cS^2(a\sw1)) \ot \pi(\cS(a\sw4) \cS^2(a\sw2)) \ot \vartheta\cS(a\sw3).
\end{align*}
Hence \eqref{NR} follows from \eqref{pair3}.

Then
\begin{align*}
\varphi(\xi(m \ot n)) &= \varphi(\xi\sw1m \ot \xi\sw2 n)&&\\
&= \varphi(\xi\swo1 \xi\swo2\sw{-1}m \ot \xi\swo2\sw0 n)&&\\
&= \varphi\left(\xi\swo1 \left(m\cS^{-1}(\xi\swo2\sw{-1})\right) \ot \xi\swo2\sw0 n\right)&&\\
&= \varphi\Big(\Big\langle \xi\swo1,\vartheta\cS\left(\xi\swo2\sw{-1} m\swe1 \cS^{-1}(\xi\swo2\sw{-3})\right) \Big\rangle m\swe0 \cS^{-1}(\xi\swo2\sw{-2}) &&\\
 &\phantom{aa}\ot \langle \xi\swo2\sw0,\vartheta\cS(n\swe1)\rangle n\swe0\Big)& &\\
&= \varphi(m\swe0 \cS^{-1}(\xi\swo2\sw{-1}) \ot n \swe0) &(\text{by } \eqref{vartheta1})&\\
&\phantom{aa}\times \Big\langle \xi\swo1,\vartheta\left(\xi\swo2\sw{-2} \cS(m\swe1)\right)\Big\rangle \langle\xi\swo2\sw{0},\vartheta\cS(n\swe1) \rangle.&&\\
\end{align*}
Hence by \eqref{NR} we obtain
\begin{align*}
\varphi(\xi(m \ot n)) &= \varphi(m\swe0 \cS^{-1}(\pi(\cS(n\swe1) n\swe5)) \ot n \swe0) & &\\
&\phantom{aa}\times \langle \xi\swo1,\vartheta\left(\pi( \cS(n\swe2) n\swe4) \cS(m\swe1)\right)\rangle \langle\xi\swo2,\vartheta\cS(n\swe3) \rangle &&\\
&=m\swe0 \pi(\cS^{-1}(n\swe6)n\swe2) \pi\cS^{-1}(n\swe1) \ot n\swe0 &(\text{by }\eqref{pair4})&\\
&\phantom{aa}\times \langle \xi, \vartheta\cS(n\swe4) \vartheta\left(\pi(\cS(n\swe3) n\swe5) \cS(m\swe1)\right)\rangle&&\\
&=m\swe0 \pi\cS^{-1}(n\swe4) \ot n\swe0 &&\\
&\phantom{aa}\times \Big\langle\xi,\vartheta\cS(n\swe2) \vartheta\left(\pi(\cS(n\swe1) n\swe3) \cS(m\swe1)\right) \Big\rangle&&\\
&=m\swe0 \pi\cS^{-1}(n\swe3) \ot n\swe0 \langle\xi,\vartheta\cS(m\swe1 \pi \cS^{-1}(n\swe2) n\swe1) \rangle,  &
\end{align*}
where the last equality follows from Lemma \ref{lem:vartheta3} and from \eqref{vartheta2}.

On the other hand
\begin{align*}
\xi \varphi(m \ot n) &= \xi (m \pi\cS^{-1}(n\swe1) \ot n\swe0)&&\\
&= \langle \xi,\vartheta\cS(\pi(n\swe4) m\swe1 \pi\cS^{-1}(n\swe2) n\swe1) \rangle m\swe0\pi\cS^{-1}(n\swe3) \ot n\swe0 &&\\
&= m\swe0 \pi\cS^{-1}(n\swe3) \ot n\swe0 \langle\xi, \vartheta\cS(m\swe1 \pi\cS^{-1}(n\swe2) n\swe1) \rangle. &(\text{by } \eqref{vartheta1})&
\end{align*}
Hence $\varphi(\xi(m \ot n)) = \xi \varphi(m \ot n)$.

It is easy to check that the diagrams \eqref{monoidal} commute for $(F,\varphi)$. Hence $(F,\ph)$ is a monoidal functor.
\end{proof}

\section{The second isomorphism}\label{sec:second}

In this section we assume that $(R, R^{\vee})$ is a dual pair of Hopf algebras in $\ydH$ with bijective antipodes and bilinear form $\langle\;,\; \rangle$. The monoidal isomorphism
$(F,\varphi) : \rydRH \to \ydRveeHrat$
of Theorem \ref{theor:first} induces by Lemma \ref{lem:center} a braided monoidal isomorphism between the centers
$$(F^{\cZ},{\varphi}^{\cZ}) : \cZ(\rydRH) \to \cZ(\ydRveeHrat).$$
Assume that $\ydRveeHrat$ is $R^{\vee} \#H$-faithful. By Propositions \ref{propo:comodulecenter} and \ref{propo:modulecenter}, the functors
\begin{align*}
\Psi : \rydRsmashrat &\to \cZ(\rydRH),\\
\Phi : \ydRveesmashrat &\to \cZ(\ydRveeHrat)
\end{align*}
are fully faithful, strict monoidal and braided. The functor $\Psi$ is defined
with respect to the Hopf algebra inclusion $\iota : H \to R \#H$.
We denote the image of $M \in \rydRsmashrat$ in $\rydRH$ defined
by restriction by $M_{\restr}$. The functor $\Phi$ is defined
with respect to the Hopf algebra projection $\pi : R^\vee \#H \to H$,
and we denote the image of $M \in \ydRveesmashrat$ in $\ydRveeHrat$
by ${}^{\pi}M$.

Our goal is to show in Theorem \ref{theor:second} that $(F,\varphi)$ induces a braided monoidal isomorphism
\begin{align*}
\rydRsmashrat \to \ydRveesmashrat.
\end{align*}
Let $G : \ydRveeHrat \to \rydRH$ be the inverse functor of the isomorphism $F$ of Theorem \ref{theor:first}. Then
$(G,\psi) : \ydRveeHrat \to \rydRH$ is a monoidal isomorphism, where $\psi$ is defined by \eqref{inversemonoidal}.
We first construct functors
$$\widetilde{F} : \rydRsmashrat \to \ydRveesmashrat, \;\widetilde{G} : \ydRveesmashrat \to  \rydRsmashrat$$ such that the diagrams
\begin{align}\label{diagramsecondF}
\begin{CD}
\rydRsmashrat  @>{\widetilde{F}}>>  \ydRveesmashrat\\
@V{\Psi}VV    @V{\Phi}VV\\
\cZ(\rydRH) @>{F^{\cZ}}>>  \cZ(\ydRveeHrat)
\end{CD}
\intertext{and}
\begin{CD}
\ydRveesmashrat  @>{\widetilde{G}}>>  \rydRsmashrat\\
@V{\Phi}VV    @V{\Psi}VV\\
\cZ(\ydRveeHrat) @>{G^{\cZ}}>>  \cZ(\rydRH)
\end{CD}\label{diagramsecondG}
\end{align}
commute.

The existence of $\widetilde{F}$ will follow from the next two lemmas.

\begin{lemma}\label{lem:second1}
Let $(F^{\cZ},{\varphi}^{\cZ}) : \cZ(\rydRH) \to \cZ(\ydRveeHrat)$ be the monoidal isomorphism induced by the isomorphism $(F,\varphi)$ of Theorem \ref{theor:first}. Let $M \in \rydRsmashrat$, and $\Psi(M)= (M_{\restr},\ga)$, where $\ga = c_M$ is defined in \eqref{defin:Psi}. Then
$$F^{\cZ} \Psi(M)= (F(M_{\restr}), \ti{\ga}),$$
and $\widetilde{\gamma}_{F(X)} : F(M_{\restr}) \ot F(X) \to F(X) \ot F(M_{\restr})$ is given by
\begin{align}
\ti{\ga}_{F(X)}(m \ot x) = x\swe0 \pi\left(\cS(x\swe1) x\swe4 m\swe1\right) \ot m\swe0 \pi\cS^{-1}(x\swe3) x\swe2\label{computeZ1}
\end{align}
for all $X \in \rydRH, x \in X$ and $m \in M$.
\end{lemma}
\begin{proof}
Let $X \in \rydRH$ with comodule structure
$$X \to X \ot R \#H,\; x \mapsto x\swe0 \ot x\swe1.$$
Recall that $\ti{\ga}_{F(X)} = {\varphi_{X,M_{\restr}}}^{-1} F(c_{M,X}) \varphi_{M_{\restr},X}$ by \eqref{defin:tilde}. It follows from the definition of $\varphi_{X,M_{\restr}}$ in Theorem \ref{theor:first} that
\begin{align}
{\varphi_{X,M_{\restr}}}^{-1}(x \ot m) = x \pi(m\swe1) \ot m\swe0 \label{computeZ2}
\end{align}
for all $x \in X, m\in M$. Hence
\begin{align*}
\tilde{\ga}_{F(X)}(m \ot x)&= {\varphi_{X,M_{\restr}}}^{-1} F(c_{M,X}) \varphi_{M_{\restr},X}(m \ot x)\\
&={\varphi_{X,M_{\restr}}}^{-1} F(c_{M,X})(m \pi \cS^{-1}(x\swe1) \ot x\swe0)\\
&={\varphi_{X,M_{\restr}}}^{-1}\left(x\swe0 \ot m \pi\cS^{-1}(x\swe2) x\swe1\right)\\
&= x\swe0 \pi\left(\cS\left(\left(\pi\cS^{-1}(x\swe2) x\swe1\right)\swe1\right) m\swe1 \left(\pi\cS^{-1}(x\swe2) x\swe1\right)\swe3\right)\\
 &\phantom{aa}\ot m\swe0 \left(\pi\cS^{-1}(x\swe2) x\swe1\right)\swe2\\
 &= x\swe0 \pi\left(\cS(\cS^{-1}(x\swe6) x\swe1) m\swe1 \cS^{-1}(x\swe4) x\swe3\right) \ot m\swe0 \pi\cS^{-1}(x\swe5) x\swe2\\
 &=x\swe0 \pi\left(\cS(x\swe1) x\swe4 m\swe1\right) \ot m\swe0 \pi\cS^{-1}(x\swe3) x\swe2.
\end{align*}
\end{proof}

In the next lemma we define a map $\delta_{\widetilde{F}(M)}$ which will be the coaction of $R \# H$ on $\widetilde{F}(M)$ in Theorem \ref{theor:second}.

\begin{lemma}\label{lem:second2}
Let $M \in \rydRsmashrat$. We denote the left $H$-comodule structure of $F(M_{\restr})$ by $M \to H \ot M,\; m \mapsto m\sw{-1} \ot m\sw0$. Define a linear map
\begin{align*}
\delta_M^{R^{\vee}} : M \to R^{\vee} \ot M,\; m \mapsto m\swos{-1} \ot m\swos0,
\end{align*}
by the equation
\begin{align}
mr = \langle r, \cS_{R^{\vee}}^{-1}\left(\cS^{-1}({m\swos0}\sw{-1}) \lact m\swos{-1}\right)\rangle' {m\swos0}\sw0\label{deltaRvee}
\end{align}
for all $r \in R,m \in M$.  Let
\begin{align}
\delta_{\widetilde{F}(M)} : M \to R^{\vee} \#H \ot M, \;m \mapsto  m\swoe{-1} \ot m\swoe0 = m\swos{-1} {m\swos0}\sw{-1} \ot {m\swos0}\sw0.\label{delta}
\end{align}
Then the following hold.
\begin{enumerate}
\item For all $m \in M,a \in R \#H$,
\begin{align*}
\langle \cS^{-1}({m\swos0}\sw{-1}) \lact m\swos{-1}, \vartheta\cS(a)\rangle {m\swos0}\sw0 = m \pi\cS^{-1}(a\sw2) a\sw1.
\end{align*}
\item Let $X \in \rydRH$, and let $\ti{\ga}_{F(X)} : F(M_{\restr}) \ot F(X) \to F(X) \ot F(M_{\restr})$ be the isomorphism in $\ydRveeH$ defined in Lemma \ref{lem:second1}. Then for all $x \in X$ and $m \in M$, $m\swoe{-1}x \ot m\swoe0 = \ti{\ga}_{F(X)}(m \ot x)$.
\item For all $m \in M$, $\pi(m\swoe{-1}) \ot m\swoe0 = m\sw{-1} \ot m\sw0$.
\end{enumerate}
\end{lemma}
\begin{proof}
(1) The map $\delta_M^{R^{\vee}}$ is well-defined since $M$ is a rational right $R$-module, $\langle\;,\;\rangle$ is non-degenerate, and the maps $\cS_{R^{\vee}}$ and
\begin{align*}
R^{\vee} \ot M \to R^{\vee} \ot M,\; \xi \ot m \mapsto \cS^{-1}({m}\sw{-1}) \lact \xi \ot {m}\sw0,
\end{align*}
are bijective.

Note that if (1) holds for $a \in R \#H$ then it holds for $ha$ for all $h \in H$. Thus it is enough to assume in (1) that $a \in \cS^{-1}(R)$. For all $r \in R$  and $a= \cS^{-1}(r)$,
\begin{align*}
\pi\cS^{-1}(a\sw2) a\sw1 = \pi\cS^{-2}(r\sw1) \cS^{-1}(r\sw2) = \cS_R(\cS^{-2}(r))
\end{align*}
by \eqref{antipode}. Therefore (1) is equivalent to
\begin{align*}
\langle \cS^{-1}({m\swos0}\sw{-1}) \lact m\swos{-1}, r\rangle {m\swos0}\sw0 = m \cS_R(\cS^{-2}(r))
\end{align*}
for all $r \in R, m\in M$. This last equation holds by our definition of $\delta_M^{R^{\vee}}$ since
\begin{align*}
m \cS_R(\cS^{-2}(r)) &= \langle \cS_R(\cS^{-2}(r)), \cS_{R^{\vee}}^{-1}\left(\cS^{-1}({m\swos0}\sw{-1}) \lact m\swos{-1}\right)\rangle' {m\swos0}\sw0&\text{(by } \eqref{deltaRvee})&\\
&=\langle \cS^{-2}(r), \cS^{-1}({m\swos0}\sw{-1}) \lact m\swos{-1}\rangle' {m\swos0}\sw0 & \text{(by } \eqref{pair6})&\\
&=\langle \cS^{-1}({m\swos0}\sw{-1}) \lact m\swos{-1}, r\rangle {m\swos0}\sw0.&&
\end{align*}
Here, we used that by Lemma \ref{lem:inversepair}, $(R^{\vee},R)$ together with $\langle \;,\; \rangle' : R \ot R^{\vee} \to \fie$ is a dual pair of Hopf algebras in $\ydH$.

(2) Let $X \in \rydRH$. By Lemma \ref{lem:second1} we have to show that
\begin{align}
m\swoe{-1} x \ot m\swoe0 =x\swe0 \pi(\cS(x\swe1) x\swe4 m\swe1) \ot m\swe0 \pi\cS^{-1}(x\swe3) x\swe2\label{second21}
\end{align}
for all $x \in X,m\in M$.

By \eqref{delta} and \eqref{deltaRvee}, the  left-hand side of \eqref{second21} can be written as
\begin{align*}
m\swoe{-1}x \ot m\swoe0 &= m\swos{-1} ({m\swos0}\sw{-1}x) \ot {m\swos0}\sw0\\
&=m\swos{-1} (x\cS^{-1}({m\swos0}\sw{-1})) \ot {m\swos0}\sw0\\
&= \langle m\swos{-1}, \vartheta\cS\left({m\swos0}\sw{-1} x\swe1 \cS^{-1}({m\swos0}\sw{-3})\right) \rangle x\swe0 \cS^{-1}({m\swos0}\sw{-2}) \ot {m\swos0}\sw0\\
&= \langle m\swos{-1}, \vartheta\cS\left( x\swe1 \cS^{-1}({m\swos0}\sw{-2})\right) \rangle x\swe0 \cS^{-1}({m\swos0}\sw{-1}) \ot {m\swos0}\sw0,
\end{align*}
where the last equality follows from \eqref{vartheta1}. Thus \eqref{second21} is equivalent to the equation
\begin{align}
&\phantom{aa.}\langle m\swos{-1}, \vartheta\cS\left( x\swe1 \cS^{-1}({m\swos0}\sw{-2})\right) \rangle x\swe0 \cS^{-1}({m\swos0}\sw{-1}) \ot {m\swos0}\sw0 \label{second22}\\
 &=x\swe0 \pi(\cS(x\swe1) x\swe4 m\swe1) \ot m\swe0 \pi\cS^{-1}(x\swe3) x\swe2\notag
\end{align}
for all $x \in X,m \in M$.

To simplify \eqref{second22} we apply the isomorphism
\begin{align}
X \ot M \to X \ot M,\; x \ot m \mapsto x \cS^{-2}(m\sw{-1}) \ot m\sw0.\label{iso}
\end{align}
Under the isomorphism \eqref{iso} the left-hand side of \eqref{second22} becomes
\begin{align*}
&\phantom{aa.}\langle m\swos{-1}, \vartheta\cS\left( x\swe1 \cS^{-1}({m\swos0}\sw{-1})\right) \rangle x\swe0  \ot {m\swos0}\sw0&&\\
&=\langle m\swos{-1}, \vartheta\left({m\swos0}\sw{-1}\cS( x\swe1)\right)\rangle x\swe0  \ot {m\swos0}\sw0&&\\
&=\langle m\swos{-1}, {m\swos0}\sw{-1} \lact \vartheta \cS( x\swe1)\rangle x\swe0  \ot {m\swos0}\sw0 & \text{(by } \eqref{vartheta2})&\\
&=\langle \cS^{-1}({m\swos0}\sw{-1}) \lact m\swos{-1}, \vartheta\cS(x\swe1)\rangle x\swe0  \ot {m\swos0}\sw0,& \text{(by } \eqref{pair2})&
\end{align*}
and the right-hand side equals
\begin{align*}
&\phantom{aa.}x\swe0 \pi\left(\cS(x\swe1) x\swe4 m\swe1\right) \cS^{-2}\cS\pi\left((m\swe0 \pi\cS^{-1}(x\swe3) x\swe2)\swe1\right) \ot (m\swe0 \pi\cS^{-1}(x\swe3) x\swe2)\swe0\\
&=x\swe0 \pi\left(\cS(x\swe1) x\swe8 m\swe2\right) \cS^{-1} \pi\left(\cS\left(\pi\cS^{-1}(x\swe7) x\swe2\right) m\swe1 \pi\cS^{-1}(x\swe5) x\swe4\right)\\
 &\phantom{aa.}\ot m\swe0 \pi \cS^{-1}(x\swe6) x\swe3\\
&= x\swe0 \pi\left(\cS(x\swe1) x\swe8 m\swe2\right) \cS^{-1}\pi\left(\cS(x\swe2)x\swe7 m\swe1 \cS^{-1}(x\swe5)x\swe4 \right)\\
&\phantom{aa.}\ot m\swe0 \pi \cS^{-1}(x\swe6) x\swe3\\
&=x\swe0 \ot m \pi\cS^{-1}(x\swe2) x\swe1.
\end{align*}
Thus the claim follows from (1).

(3) Let $m \in M$. By \eqref{deltaRvee} and \eqref{pair4}, $m = m1 = \varepsilon(m\swos{-1}) m\swos0$. Hence
$$\pi(m\swoe{-1}) \ot m\swoe0 = \varepsilon(m\swos{-1}) {m\swos0}\sw{-1}  \ot {m\swos0}\sw0 = m\sw{-1} \ot m\sw0.$$
\end{proof}

The existence of $\widetilde{G}$ will follow from the next two lemmas.

Let $M \in {\ydRveesmashrat}$. We denote the left $R^{\vee} \# H$-comodule structure of $M$ by
$$M \to R^{\vee} \# H \ot M, \; m \mapsto m\swoe{-1} \ot m\swoe0 = m\swos{-1}{m\swos0}\sw{-1} \ot {m\swos0}\sw0,$$
where $M \to R^{\vee} \ot M,\; m \mapsto m\swos{-1} \ot m\swos0,$ is the $R^{\vee}$-comodule structure of $M$.
For all $X \in {\ydRveeHrat}$ the right $R \# H$-comodule structure of $G(X)$  is denoted by
$$X \to X \ot R \# H,\; x \mapsto x\swe0 \ot x\swe1.$$
Note that $G(X) = X$ as a vector space. The right $H$-module structure of $G(X)$ is defined by
\begin{align}
xh = \cS(h)x\label{ruleX}
\end{align}
for all $x \in X, h \in H$. Since $FG({}^{\pi}M) ={}^{\pi}M$, it follows that
\begin{align}
\pi\cS(m\swe1) \ot m\swe0 = \pi(m\swoe{-1}) \ot m\swoe0\label{ruleM}
\end{align}
for all $m \in M$.

\begin{lemma}\label{lem:second3}
Let $(G^{\cZ},\psi^{\cZ}) : \cZ(\ydRveeHrat) \to \cZ(\rydRH)$ be the monoidal isomorphism induced by the monoidal isomorphism $(G,\psi)$. Let $M \in \ydRveesmashrat$, and $\Phi(M) = (^{\pi}M, \gamma)$, where $\gamma = c_M$ is defined in \eqref{defin:Phi}.
Then
$$G^{\cZ} \Phi(M) = (G(^{\pi}M), \widetilde{\gamma}),$$
and $\widetilde{\gamma}_{G(X)} : G({}^{\pi}M) \ot G(X) \to G(X) \ot G({}^{\pi}M)$ is given by
\begin{align}
\widetilde{\gamma}_{G(X)}(m \ot x) = \left(\cS^{-1}({m\swos{0}}\sw{-1} \pi\cS^2(x\swe1)) \lact m\swos{-1}\right) x\swe0 \ot \pi \cS(x\swe2) {m\swos0}\sw0\label{resulttilde}
\end{align}
for all $X \in {\ydRveeHrat}, x \in X$ and $m \in M$.
\end{lemma}
\begin{proof}
Let $X \in {\ydRveeHrat}$. By \eqref{defin:Phi}, $\gamma_X : {}^{\pi}M \ot X \to X \ot {^{\pi}M}$ is defined by
$$\gamma_X(m \ot x) = m\swoe{-1} x \ot m\swoe0$$
for all $x \in X,m \in M$.

By \eqref{inversemonoidal} and \eqref{defin:tilde}, the isomorphism $\widetilde{\gamma}_{G(X)}$ is defined by the equation
\begin{align}
\widetilde{\gamma}_{G(X)}G(\ph_{G({}^{\pi}M),G(X)})  = G(\ph_{G(X),G({}^{\pi}M)})G(\gamma_X).\label{equationtilde}
\end{align}
We apply both sides of \eqref{equationtilde} to an element $m \ot x, m \in M,x \in X$. Then
\begin{align*}
\widetilde{\gamma}_{G(X)}G(\ph_{G({}^{\pi}M),G(X)})(m \ot x) &= \widetilde{\gamma}_{G(X)}(m \pi\cS^{-1}(x\swe1) \ot x\swe0),&&\\
\intertext{and}
G(\ph_{G(X),G({}^{\pi}M)})G(\gamma_X)(m \ot x) &= (m\swoe{-1} x) \pi\cS^{-1}({m\swoe0}\swe1) \ot {m\swoe0}\swe0&&\\
&=\pi({m\swoe0}\swe1) (m\swoe{-1} x) \ot {m\swoe0}\swe0 &&(\text{by } \eqref{ruleX})\\
&=\pi\cS^{-1}(m\swoe{-1}) m\swoe{-2}x \ot m\swoe0 && (\text{by } \eqref{ruleM})\\
&=\left(\cS^{-1}({m\swos0}\sw{-1}) \lact m\swos{-1}\right)x \ot {m\swos0}\sw0,
\end{align*}
where in the proof of the last equality the following formula in $R^{\vee} \# H$ is used for $a = m\swoe{-1} = m\swos{-1} {m\swos0}\sw{-1}$. Let $\xi \in R^{\vee},h \in H$ and $a=\xi h \in R^{\vee} \# H$. Then
\begin{align*}
\pi\cS^{-1}(a\sw2)a\sw1 &= \pi\cS^{-1}({\xi\swo2}\sw0 h\sw2) \xi\swo1 {\xi\swo2}\sw{-1} h\sw1\\
&= \cS^{-1}(h\sw2) \xi h\sw1\\
&= \cS^{-1}(h) \lact \xi.
\end{align*}
We have shown that
\begin{align}
\widetilde{\gamma}_{G(X)}(m \pi\cS^{-1}(x\swe1) \ot x\swe0)= \left(\cS^{-1}({m\swos0}\sw{-1}) \lact m\swos{-1}\right)x \ot {m\swos0}\sw0.\label{formulatilde}
\end{align}
Since $m \ot x = m \pi(x\swe2) \pi\cS^{-1}(x\swe1) \ot x\swe0 = (\pi\cS(x\swe2)m) \pi\cS^{-1}(x\swe1) \ot x\swe0$, we obtain from \eqref{formulatilde} and the Yetter-Drinfeld condition for $M$
\begin{align*}
\widetilde{\gamma}_{G(X)}&(m \ot x) \\
&=\left(\cS^{-1}\left({(\pi\cS(x\swe1)m)\swos0}\sw{-1}\right) \lact \left(\pi\cS(x\swe1)m\right)\swos{-1}\right)x\swe0
\ot {(\pi\cS(x\swe1)m)\swos0}\sw0\\
&=\left(\cS^{-1}\left((\pi\cS(x\swe1)m\swos0)\sw{-1}\right) \lact \pi\cS(x\swe2) m\swos{-1}\right) x\swe0 \ot \left(\pi\cS(x\swe1) m\swos0\right)\sw0\\
&=\left(\cS^{-1}({m\swos{0}}\sw{-1} \pi\cS^2(x\swe1)) \lact m\swos{-1}\right) x\swe0 \ot \pi \cS(x\swe2) {m\swos0}\sw0.
\end{align*}
\end{proof}

\begin{lemma}\label{lem:second4}
Let $M \in \ydRveesmashrat$, and $G^{\cZ} \Phi(M) = (G(^{\pi}M), \widetilde{\gamma})$ as in Lemma \ref{lem:second3}. Let $\langle\;,\; \rangle' : R \ot R^{\vee} \to \fie$ be the form defined in \eqref{defin:inversepair}. Define a linear map
$\mu_{\widetilde{G}(M)} : M \ot R \# H \to M$
by
\begin{align}
\mu_{\widetilde{G}(M)}(m \ot a)= \langle m\swos{-1},{m\swos0}\sw{-1} \lact \vartheta(\pi\cS^2(a\sw2) \cS(a\sw1)) \rangle \pi\cS(a\sw3) {m\swos0}\sw0 \label{tildeaction}
\end{align}
for all $m \in M,a \in R \# H$. Then the following hold.
\begin{enumerate}
\item For all $X \in {\ydRveeHrat}, x \in X$ and $m \in M$,
$$\widetilde{\gamma}_{G(X)} (m \ot x) = x\swe0 \ot \mu_{\widetilde{G}(M)}(m \ot x\swe1).$$
\item For all $m \in M$ and $h \in H$, $\mu_{\widetilde{G}(M)}(m \ot h) = mh$.
\item For all $m \in M$ and $r \in R$, $\mu_{\widetilde{G}(M)}(m \ot r)=
\langle r, \vartheta\cS^{-1}(m\swoe{-1}) \rangle' m\swoe0$.
\end{enumerate}
\end{lemma}
\begin{proof}
(1) Let $X \in {\ydRveeHrat}, x \in X$ and $m \in M$. Then by \eqref{resulttilde},
\begin{align*}
\widetilde{\gamma}_{G(X)}(m \ot x) &= \left(\cS^{-1}\left({m\swos{0}}\sw{-1} \pi\cS^2(x\swe1)\right) \lact m\swos{-1}\right) x\swe0 \ot \pi \cS(x\swe2) {m\swos0}\sw0 \\
&= \langle \cS^{-1}\left({m\swos0}\sw{-1} \pi\cS^2(x\swe2)\right) \lact m\swos{-1}, \vartheta \cS(x\swe1) \rangle x\swe0 \ot \pi \cS(x\swe3) {m\swos0}\sw0\\
&=x\swe0 \ot \langle m\swos{-1}, {m\swos0}\sw{-1} \lact \vartheta\left(\pi\cS^2(x\swe2) \cS(x\swe1)\right) \rangle \pi\cS(x\swe3) {m\swos0}\sw0\\
&= x\swe0 \ot \mu_{\widetilde{G}(M)}(m \ot x\swe1),
\end{align*}
where we used \eqref{first3} and the equality $X = FG(X)$ together with \eqref{vartheta2} and \eqref{pair2}.

(2) Let $m \in M$ and $h \in H$. Then
\begin{align*}
\mu_{\widetilde{G}(M)}(m \ot h) &= \langle m\swos{-1},{m\swos0}\sw{-1} \lact \vartheta\left(\pi\cS^2(h\sw2) \cS(h\sw1)\right) \rangle \pi\cS(h\sw3) {m\swos0}\sw0&&\\
&= \langle m\swos{-1},{m\swos0}\sw{-1} \lact 1 \rangle \pi\cS(h) {m\swos0}\sw0&&\\
&= \langle m\swos{-1},1 \rangle \pi\cS(h) m\swos0&& \text{(by } \eqref{pair5})\\
&= \pi\cS(h) m&&\\
&= mh.&&
\end{align*}

(3) Let $m \in M$ and $r \in R$. Then $r\sw1 \ot \pi(r\sw2) = r \ot 1$. Hence
\begin{align*}
\mu_{\widetilde{G}(M)}(m \ot r)&= \langle m\swos{-1}, {m\swos0}\sw{-1} \lact \vartheta(\cS(r)) \rangle {m\swos0}\sw0\\
&=\langle m\swos{-1}, {m\swos0}\sw{-1} \lact \vartheta(\cS(r\sw{-1}) \cS_R(r\sw0)) \rangle {m\swos0}\sw0&&\text{(by } \eqref{bigS})\\
&=\langle m\swos{-1}, \left({m\swos0}\sw{-1}\cS(r\sw{-1})\right) \lact \cS_R(r\sw0) \rangle {m\swos0}\sw0 &&\text{(by } \eqref{vartheta2})\\
&= \langle \cS^{-1}({m\swos0}\sw{-1}) \lact m\swos{-1}, \cS(r\sw{-1}) \lact \cS_R(r\sw0) \rangle {m\swos0}\sw0&&\text{(by } \eqref{pair2})\\
&= \langle \cS^{-1}({m\swos0}\sw{-1}) \lact \cS_{R^{\vee}}^{-1}(m\swos{-1}), \cS(r\sw{-1}) \lact \cS_R^2(r\sw0) \rangle {m\swos0}\sw0&&\text{(by } \eqref{pair6})\\
&=\langle \cS^{-1}({m\swos0}\sw{-1}) \lact \cS_{R^{\vee}}^{-1}(m\swos{-1}), \cS^2(r) \rangle {m\swos0}\sw0&&\text{(by } \eqref{bigS2})\\
&=\langle \cS^{-1}({m\swos0}\sw{-1}) \lact \vartheta\cS^{-1}(m\swos{-1}), \cS^2(r) \rangle {m\swos0}\sw0&&\text{(by } \eqref{ruleS5})\\
&=\langle \vartheta\cS^{-1}(m\swos{-1}{m\swos0}\sw{-1}), \cS^2(r) \rangle {m\swos0}\sw0&&\text{(by } \eqref{vartheta2})\\
&=\langle r, \vartheta\cS^{-1}(m\swoe{-1}) \rangle' m\swoe0.
\end{align*}
\end{proof}

\begin{theor}\label{theor:second}
Let $(R, R^{\vee})$ be a dual pair of Hopf algebras in $\ydH$ with bijective antipodes and bilinear form $\langle\;,\; \rangle : R^{\vee} \ot R \to \fie$. Let $\langle\;,\; \rangle' : R \ot R^{\vee} \to \fie$ be the form defined in \eqref{defin:inversepair}. Assume that $\ydRveeHrat$ is $R^{\vee} \#H$-faithful.

Then the functor
$$(\widetilde{F},\widetilde{\varphi}) : \rydRsmashrat \to \ydRveesmashrat$$
as defined below is a braided monoidal isomorphism.

For any object $M \in \rydRsmashrat$ with right $R \# H$-comodule structure denoted by
$$\delta_M: M \to M \ot R \# H,\; m \mapsto m\swe0 \ot m\swe1,$$
let $\widetilde{F}(M)=M$ as a vector space and $\widetilde{F}(M) \in \ydRveesmashrat$
with left $H$-action, $H$-coaction $\delta_{\widetilde{F}(M)}^H$, $R^{\vee}$-action, and $R^{\vee} \#H$-coaction
$$\delta_{\widetilde{F}(M)} : M \to R^{\vee} \#H \ot M,\; m \mapsto m\swoe{-1} \ot m\swoe0,$$
respectively, given by
\begin{align}
hm &= m \cS^{-1}(h),\label{tilde1}\\
\delta_{\widetilde{F}(M)}^H(m) &= \pi\cS(m\swe1) \ot m\swe0,\label{tilde2}\\
\xi m &= \langle \xi,\vartheta\cS(m\swe1)\rangle m\swe0,\label{tilde3}\\
mr&= \langle r, \vartheta\cS^{-1}(m\swoe{-1})\rangle' m\swoe0\label{tilde4}
\end{align}
for all $h \in H,m \in M, \xi \in R^{\vee}$ and $r \in R$. For any morphism $f$ in $\rydRsmashrat$ let $\widetilde{F}(f) = f$.
The natural transformation $\widetilde{\varphi}$ is defined by
\begin{align}
\widetilde{\varphi}_{M,N} : \widetilde{F}(M) \ot \widetilde{F}(N) &\to \widetilde{F}(M \ot N),\\
m \ot n &\mapsto m \pi\cS^{-1}(n\swe1) \ot n\swe0 = \pi \cS^{-1}(n\swoe{-1})m \ot n\swoe0,
\end{align}
for all $M,N \in \rydRsmashrat$.
\end{theor}
\begin{proof}
Let $M \in \rydRsmashrat$. As in Lemma \ref{lem:second1} we write $\Psi(M)= (M_{\restr},\ga)$. Then
$$F^{\cZ} \Psi(M)= (F(M_{\restr}), \ti{\ga}).$$
By Lemma \ref{lem:reformulation}, the definitions of $\delta_{\widetilde{F}(M)}$ in Lemma \ref{lem:second2} and in \eqref{tilde4} coincide. Thus, by Lemma \ref{lem:second2} (2), for all $X \in \rydRH$, the isomorphism
$$\ti{\ga}_{F(X)} : F(M_{\restr}) \ot F(X) \to F(X) \ot F(M_{\restr})$$
has the form
$$\ti{\ga}_{F(X)}(m \ot x) = m\swoe{-1}x \ot m\swoe0$$
 for all $m \in M,x \in X$, where $\delta_{\widetilde{F}(M)}(m) = m\swoe{-1} \ot m\swoe0$ is defined in Lemma \ref{lem:second2}. By Lemma \ref{lem:second2} (3), the left $H$-comodule structure of $F(M_{\restr})$ is $(\pi \ot \id) \delta_{\widetilde{F}(M)}$.
The left $H$-action, $H$-coaction and $R^{\vee} \#H$-action of $\widetilde{F}(M)$ are those of $F(M_{\restr})$, see Theorem \ref{theor:first}.

We now conclude from Proposition \ref{propo:modulecenter} that $\widetilde{F}(M)$ with $R^{\vee} \#H$-comodule structure $\delta_{\widetilde{F}(M)}$ is an object in $\rydRsmashrat$, and $\Phi(\widetilde{F}(M)) = F^{\cZ}\Psi(M)$.

Thus we have defined a functor $\widetilde{F} : \rydRsmashrat \to \ydRveesmashrat$ such that the diagram \eqref{diagramsecondF} commutes. By Lemma \ref{lem:diagram} there is a uniquely determined family $\widetilde{\varphi}$ such that $(\ti{F},\ti{\varphi})$ is a braided monoidal functor with
$$(F^{\cZ},\varphi^{\cZ})(\Psi,\id) = (\Phi,\id)(\ti{F},\ti{\varphi}).$$
Let $M,N \in \rydRsmashrat$. Then $\Phi(\ti{\varphi}_{M,N}) = \varphi^{\cZ}_{\Psi(M),\Psi(N)}$ by \eqref{monoidalcomposition}, that is, for all $m \in M,n \in N$,
$$\ti{\varphi}_{M,N}(m \ot n)=\varphi_{M_{\restr},N_{\restr}}(m \ot n) = m\pi\cS^{-1}(n\swe1) \ot n\swe0$$
by Theorem~\ref{theor:first}.
To define the inverse functor of $\widetilde{F}$ let $M \in \ydRveesmashrat$.
Let $\widetilde{G}(M) = M$ as a vector space with right $R \# H$-comodule structure
and $H$-module structure given by ${}^{\pi}M$, and with right $R \# H$-module structure
$\mu_{\widetilde{G}}$ defined in \eqref{tildeaction}.
Then $\widetilde{G}(M) \in \rydRsmash$ by Proposition \ref{propo:comodulecenter} and
Lemma \ref{lem:second4} (1), (2). It follows from Lemma \ref{lem:second4} (3)
that $\widetilde{G}(M)$ is rational as an $R$-module. We let $\widetilde{G}(f) =f$
for morphisms in $\ydRveesmashrat$.

Thus we have defined a functor $\widetilde{G} : \ydRveesmashrat \to \rydRsmashrat$,
and it is clear form the explicit definitions of $\widetilde{F}$ and $\widetilde{G}$
that $\widetilde{F}\widetilde{G} = \id$ and $\widetilde{G}\widetilde{F} = \id$.
\end{proof}

\section{The third isomorphism}\label{sec:third}

Finally we compose the isomorphism in Theorem \ref{theor:second} with the isomorphism in Lemma \ref{lem:leftright}.

We recall from Lemma \ref{lem:Rmod} and Lemma \ref{lem:Rcomod} the description of left modules and left comodules over $R \# H$, where $R$ is a Hopf algebra in $\ydH$. In particular, the restriction of an object $M \in \ydRsmash$ with $R \# H$-comodule structure $\delta_M$ is an object in $\ydH$, where the $H$-action is defined by restriction and the $H$-coaction is $(\pi \ot \id)\delta_M$.
\begin{theor}\label{theor:third}
Let $(R, R^{\vee})$ be a dual pair of Hopf algebras in $\ydH$  with bijective antipodes and with bilinear form $\langle\;,\; \rangle : R^{\vee} \ot R \to \fie$. Assume that $\ydRveeHrat$ is $R^{\vee} \#H$-faithful.

Then the functor
$$(\Fu,\fu) : \ydRsmashrat \to \ydRveesmashrat$$
as defined below is a  braided monoidal isomorphism.

Let $M \in \ydRsmashrat$ with left  $R$-comodule structure denoted by
$$\delta_M^R  : M \to R \ot M,\; m \mapsto m\sws{-1} \ot m\sws0.$$
Let $\Fu(M)=M$ as an object in $\ydH$ by restriction, and $\Fu(M) \in \ydRveesmashrat$
with $R^{\vee}$-action and $R^{\vee}$-coaction $\delta_{\Fu(M)}^{R^{\vee}}$,
respectively, given by
\begin{align}
\xi m &= \langle \xi,m\sws{-1}\rangle m\sws0,\label{actions}\\
\delta_{\Fu(M)}^{R^{\vee}} (m) &= c^2_{R^{\vee},M}(m\swoss{-1} \ot m\swoss0),\label{coactionss}\\
\intertext{where}
rm &= \langle m\swoss{-1}, \theta_R(r) \rangle m\swoss0\label{coactionss1}
\end{align}
for all $m \in M, \xi \in R^{\vee}$ and $r \in R$. For any morphism $f$ in $\ydRsmashrat$ let $\Fu(f)=f$.
The natural transformation $\fu$ is defined by
\begin{align}
\fu_{M,N} : \Fu(M) \ot \Fu(N) \to \Fu(M \ot N),\; m \ot n \mapsto \cS^{-1} \cS_R(n\sws{-1})  m \ot n \sws0,\label{fu}
\end{align}
for all $M,N \in \ydRsmashrat$.
\end{theor}
\begin{proof}
Let $(S_1^{-1}, \psi) : \ydRsmash \to \rydRsmash$ be the braided monoidal isomorphism
defined in Lemma \ref{lem:leftright} (2).
Let $M \in \ydRsmash$, and assume that $M$ is rational as a left $R$-module.
By definition, $S_1^{-1}(M) =M$ as a vector space, and $mr = \cS(r)m$ for all $m \in M,r \in R$,
where $\cS$ is the antipode of $R \# H$.
Let $m \in M$. Since $M$ is a rational left $R$-module, $E'^{\perp}m =0$ for some $E'\in \cE_{R^{\vee}}$.
Choose a subspace $E''\in \cE_{R^{\vee}}$ with $\cS_{R^{\vee}}(E') \sub E''$.
Then  $\cS(r)m=\cS(r\sw{-1}) \cS_R(r\sw0)m=0$ for all $r \in E''^{\perp}$ by \eqref{bigS} and \eqref{pair6}.
Hence $S_1^{-1}(M)$ is rational as a right $R$-module.

Thus $(S_1^{-1}, \psi)$ induces a functor  on the rational objects. We denote the induced functor again by
$$(S_1^{-1}, \psi) : \ydRsmashrat \to \rydRsmashrat.$$
Let
$$(\widetilde{F},\widetilde{\varphi}) : \rydRsmashrat \to \ydRveesmashrat$$
 be the braided monoidal isomorphism of Theorem \ref{theor:second}. Then the composition
\begin{align}\label{def:fu}
(\Fu,\fu) =(\widetilde{F},\widetilde{\varphi})(S_1^{-1}, \psi)
\end{align}
is a braided monoidal isomorphism.

 Let $M \in \ydRsmashrat$. The $R^{\vee} \# H$-coaction denoted by
 $$\delta_{\Fu(M)} : M \to R^{\vee} \# H \ot M,\; m \mapsto m\swoe{-1} \ot m\swoe0,$$
 is given by
\begin{align}
\cS(r)m= \langle r, \vartheta\cS^{-1}(m\swoe{-1})\rangle' m\swoe0\label{coactions}
\end{align}
for all $m \in M$ and $r \in R$.

Let
$$\delta_{\Fu(M)}^{R^{\vee}} = (\vartheta \ot \id)\delta_{\Fu(M)} : M \to R^{\vee} \ot M,\;m \mapsto m\swos{-1} \ot m\swos0,$$
be the $R^{\vee}$-coaction of $\Fu(M)$.

To prove \eqref{coactionss}, let $m \in M, r \in R$. Then by \eqref{coactions} and \eqref{bigS2},
\begin{align*}
\cS(r)m &= \langle \vartheta \cS^{-1}(m\swoe{-1}), \cS_R^2(\theta_R(r)) \rangle m\swoe0,
\end{align*}
hence
\begin{align*}
\cS_R(r) m &= \langle \vartheta \cS^{-1}(m\swoe{-1}), \cS^2_R(\theta_R(r\sw0) \rangle r\sw{-1} m\swoe0 &&\text{(by }\eqref{bigS})\\
&=\langle \cS^2_{R^{\vee}} \vartheta \cS^{-1}(m\swos{-1} {m\swos0}\sw{-1}),\theta_R(r\sw0) \rangle r\sw{-1} {m\swos0}\sw0 &&\text{(by }\eqref{pair6})\\
&=\langle \cS^{-1}({m\swos0}\sw{-1}) \lact \cS^2_{R^{\vee}} \vartheta \cS^{-1}(m\swos{-1}) ,\theta_R(r\sw0) \rangle r\sw{-1} {m\swos0}\sw0 &&\text{(by }\eqref{vartheta2})\\
&=\langle \cS^{-1}({m\swos0}\sw{-1}) \lact \cS_{R^{\vee}} (m\swos{-1}) ,\theta_R(r\sw0) \rangle r\sw{-1} {m\swos0}\sw0 &&\text{(by }\eqref{ruleS5})\\
&=\langle \cS^{-1}({m\swos0}\sw{-1}) \lact \cS_{R^{\vee}} (m\swos{-1}) ,\theta_R(r)\sw0 \rangle \cS^{-2}(\theta_R(r)\sw{-1}){m\swos0}\sw0 &&\text{(by }\eqref{theta2}).
\end{align*}
Since $\theta_R \cS_R^{-1} = \cS_R^{-1} \theta_R$, we obtain by \eqref{pair6}
\begin{align}\label{formulathird}
rm &= \langle \cS^{-1}({m\swos0}\sw{-1}) \lact m\swos{-1}, \theta_R(r)\sw0 \rangle \cS^{-2}(\theta_R(r)\sw{-1}) {m\swos0}\sw{0}.
\end{align}
Note that $c^{-1}_{R^{\vee},M}(m\swos{-1} \ot m\swos0) = {m\swos0}\sw0 \ot \cS^{-1}({m\swos0}\sw{-1}) \lact m\swos{-1}$. Hence by \eqref{formulathird} and \eqref{pair3},
\begin{align*}
rm &= \langle m\swoss{-1}, \theta_R(r) \rangle m\swoss0,
\end{align*}
where $m\swoss{-1} \ot m\swoss0 = c_{M,R^{\vee}}^{-1} c_{R^{\vee},M}^{-1}(m\swos{-1} \ot m\swos0)$.

Finally, by \eqref{def:fu} and \eqref{monoidalcomposition} the natural transformation $\fu$ is given by
\begin{align}\label{fuold}
\fu_{M,N} : \Fu(M) \ot \Fu(N) \to \Fu(M \ot N),\; m \ot n \mapsto n\swe{-1} \pi \cS^{-1}(n\swe{-2}) m \ot n \swe0,
\end{align}
for all $M,N \in \ydRsmashrat$, where
$$N \to R \# H \ot N,\; n \mapsto n\swe{-1} \ot n\swe0 = n\sws{-1} {n\sws0}\sw{-1} \ot {n\sws0}\sw0,$$
denotes the $ R \# H$- coaction of $N$. Let $r \in R, h \in H$ and $a = rh \in R \# H$. Then
\begin{align*}
a\sw2 \pi\cS^{-1}(a\sw1) &= \varepsilon(h) r\sw2 \pi\cS^{-1}(r\sw1)&&\\
&= \varepsilon(h) {r\swo2}\sw0 \pi\cS^{-1}(r\swo1 {r\swo2}\sw{-1})&&\\
&= \varepsilon(h) r\sw0 \cS^{-1}(r\sw{-1})&&\\
&= \varepsilon(h) \cS^{-1}\cS_R(r). &&\text{by \eqref{antipode}}
\end{align*}
Hence \eqref{fu} follows from \eqref{fuold}.
\end{proof}

We specialize the last theorem to the case of $\no$-graded dual pairs of braided Hopf algebras in $\ydH$.

Let $R = \oplus_{n\geq 0} R(n)$ be an $\no$-graded Hopf algebra in $\ydH$. We view the bosonization $R \# H$ as an $\no$-graded Hopf algebra with $\deg R(n)=n$  for all $n\geq 0$, and $\deg H =0$.

For any Yetter-Drinfeld module $W \in \ydRsmash$ we define two ascending filtrations of Yetter-Drinfeld modules in $\ydH$ by
\begin{align}
\Fd n W&= \{ w \in W \mid \delta^R_W(w) \in \oplus_{i=0}^n R(i) \ot W \},\\
\Fm n W&= \{ w \in W \mid R(i) w = 0 \text{ for all } i > n \}
\end{align}
for all $n \geq 0$. Then
$\cup_{n \geq 0} \Fd n W = W$. But in general,
$\cup_{n \geq 0} \Fm n W \neq W$.

Given an abelian monoid $\Gamma $ and a $\Gamma $-graded Hopf algebra $A$
with bijective antipode, we say that $M \in \ydA $ is $\Gamma $-\textit{graded}
if $M=\oplus_{\gamma \in \Gamma }M(\gamma )$
is a vector space grading and if the module and comodule maps of $M$ are $\Gamma $-graded
of degree $0$.

\begin{corol}\label{cor:third}
Let $R^{\vee}= \oplus_{n \geq 0} R^{\vee}(n)$ and $R= \oplus_{n \geq 0} R(n)$ be $\no$-graded Hopf algebras in $\ydH$ with finite-dimensional components $R^{\vee}(n)$ and $R(n)$ for all $n \geq 0$, and let $\langle \;,\; \rangle : R^{\vee} \ot R \to \fie$ be a  bilinear form of vector spaces satisfying \eqref{pair1} -- \eqref{pair5} and \eqref{gradedpair}. Then the functor
$$(\Fu,\fu) : \ydRsmashrat \to \ydRveesmashrat$$
as defined in Theorem \ref{theor:third} is a  braided monoidal isomorphism.

Moreover, the following hold.

\begin{enumerate}
\item A left $R$- (respectively $R^{\vee}$)-module $M$ is rational if and only if for any $m \in M$ there is a natural number $n_0$ such that $R(n)m =0$ (respectively $R^{\vee}(n)m =0$) for all $n \geq n_0$.
\item Let $M \in \ydRsmashrat$ be $\mZ$-graded. Then $\Omega(M)$ is a $\mZ$-graded object in $\ydRveesmashrat$ with $\Omega(M)(n) = M(-n)$ for all $m \in \mZ$.
\item For any $M \in \ydRsmashrat$ and $n \geq 0$,
\begin{align*}
\Fm n \Fu(W) = \Fd n W,\;
\Fd n \Fu(W) = \Fm n W.
\end{align*}

\end{enumerate}
\end{corol}
\begin{proof}
By Example \ref{exa:gradedpair}, the antipodes of $R$ and of $R^{\vee}$ are bijective,
and $(R, R^{\vee})$ together with $\langle\;,\; \rangle$ is a dual pair of Hopf algebras in $\ydH$.
By Example \ref{exa:regularrep} (2), the category $\ydRveeHrat$ is $R^{\vee} \#H$-faithful.
Thus $(\Fu,\fu)$ is a braided monoidal isomorphism by Theorem \ref{theor:third}.

(1) is clear from Example \ref{exa:gradedpair}, and (2) and (3)  can be checked using \eqref{actions} and \eqref{coactionss}.
\end{proof}

\begin{propo}\label{propo:Fdm}
Let $R = \oplus_{n\geq 0} R(n)$ be an $\no$-graded Hopf algebra in $\ydH$ with finite-dimensional components $R(n)$ for all $n\geq 0$. Let $W$ be an irreducible object in the category of  $\mathbb{Z}$-graded left Yetter-Drinfeld modules over $R \# H$. Assume that $W$ is locally finite as an $R$-module, or equivalently finite-dimensional.
Let $n_0 \leq n_1$ in $\ndZ$, and  $W = \oplus_{i=n_0}^{n_1} W(i)$ be the decomposition into homogeneous components such that $W( n_0) \neq 0, W(n_1) \neq 0$.
Then
\begin{align}\label{Fdm}
\Fd n W = \mathop{\oplus} _{i=n_0}^{n_0 + n} W(i),\quad \Fm n W = \mathop{\oplus} _{i=n_1 -n}^{n_1} W(i)
\end{align}
 for all $n \geq 0$.
Moreover, $W(n_0)$ and $W(n_1)$ are irreducible Yetter-Drinfeld modules over $R(0)\#H$,
where the action and coaction arise from the action and coaction of $R\#H$ on $W$
by restriction and projection, respectively.
\end{propo}

 \begin{proof}
The inclusions $\supseteq$ in \eqref{Fdm} follow from the definitions
since $W$ is a $\ndZ$-graded Yetter-Drinfeld module. On the other hand, assume that
$\Fd n W \neq \mathop{\oplus} _{i=n_0}^{n_0 + n} W(i)$ for some $n \geq 0$. Then there exist $l > n_0 + n$ and $w \in W(l) \cap \Fd n (W)$ with $w \neq 0$, since $W$ is a $\ndZ$-graded Yetter-Drinfeld module. Then the Yetter-Drinfeld submodule of $W$ generated by $w$ is contained in $\oplus_{n > n_0} W(n)$. This is a contradiction to $W(n_0) \neq 0$ and the irreducibility of $W$. The proof of the second equation in \eqref{Fdm} is similar.
By degree reasons, $W(n_0)$ is a Yetter-Drinfeld module over $R(0)\#H$
in the way explained in the claim. It is irreducible, since $W$ is irreducible
and hence it is the $R\#H$-module generated by any nonzero Yetter-Drinfeld submodule
over $R(0)\#H$ of $W(n_0)$.
Similarly, $W(n_1)$ is an irreducible Yetter-Drinfeld module over $R(0)\#H$,
since $W$ is the $R\#H$-comodule generated by any nonzero Yetter-Drinfeld submodule
over $R(0)\#H$ of $W(n_1)$.
\end{proof}

Let $R$ be a braided Hopf algebra in $\ydH$, and let $K$ be a Hopf algebra in
$\ydRsmash$. Then
\begin{align*}
K \# R := (K \# (R \# H))^{\co H}
\end{align*}
denotes the braided Hopf algebra in $\ydH$ of $H$-coinvariant elements with
respect to the canonical projection $K \# (R \# H) \to R \# H \to H$.

\begin{corol} \label{corol:Hopfproj}
In the situation of Theorem \ref{theor:third} assume that $R $ is a Hopf subalgebra  of a Hopf algebra $B$ in $\ydH$ with a Hopf algebra projection onto $R$, and let $K:=B^{\co R}$.
\begin{enumerate}
\item $K= (B \# H)^{\co R \# H}$ is a Hopf algebra in $\ydRsmash $, and the multiplication map $K \# R \to B$ is an isomorphism of Hopf algebras in $\ydH$. \item Assume that $K$ is rational as an $R$-module. Then $\Fu(K) \# R^{\vee}$ is a Hopf algebra in $\ydH$ with a Hopf algebra projection onto $R^{\vee}$.
\end{enumerate}
\end{corol}
\begin{proof}
(1) is shown in  \cite[Lemma 3.1]{a-AHS10}. By Theorem \ref{theor:third}, $\Fu(K)$ is a Hopf algebra in $\ydRveesmash $. This proves (2).
\end{proof}

\section{An application to Nichols algebras}\label{sec:Nichols}

In the last section we want to apply the construction in Corollary~\ref{corol:Hopfproj}
to Nichols algebras. We show in Theorem~\ref{theor:NtoN} that if $B$
is a Nichols algebra of a semisimple Yetter-Drinfeld module,
then the Hopf algebra $\Fu(K) \# R^{\vee}$ constructed in Corollary~\ref{corol:Hopfproj} is again a Nichols algebra.
The advantage of the construction is that the new Nichols algebra is usually not
twist equivalent to the original one.

 We start with some
general observations.

\begin{remar}\label{rem:primitive}
Let $R = \oplus_{n \in \no} R(n)$ be an $\no$-graded bialgebra in $\ydH$.

(1) The space
$$P(R) = \{ x \in R \mid \Delta_R(x) = 1 \ot x + x \ot 1 \}$$
of primitive elements of $R$ is an $\no$-graded subobject of $R$ in $\ydH$, since it is the kernel of the graded, $H$-linear and $H$-colinear map
$$R \to R \ot R,\; x \mapsto \Delta_R(x) - 1 \ot x - x \ot 1.$$

(2) Assume that $R(0) = \fie$. Then $R(1) \sub P(R)$. Moreover, $R$ is an $\no$-graded braided Hopf algebra in $\ydH$.
\end{remar}

Let $M \in \ydH$. A {\em pre-Nichols algebra} \cite{a-Masuo08} of $M$ is an $\no$-graded braided bialgebra $R = \oplus_{n \geq \no} R(n)$ in $\ydH$ such that
\begin{enumerate}
\item [(N1)]$R(0) = \fie$,
\item [(N2)]$R(1) = M$,
\item [(N3)]$R$ is generated as an algebra by $M$.
\end{enumerate}
The Nichols algebra of $M$ is a pre-Nichols algebra $R$ of $M$ such that
\begin{enumerate}
\item [(N4)]$P(R) \cap R(n) = 0$ for all $n \geq 2$.
\end{enumerate}
It is denoted by $\NA(M)$. Up to isomorphism, $\NA(M)$ is uniquely determined by $M$. By Remark \ref{rem:primitive}, our definition of $\NA(M)$ coincides with \cite[Def.~2.1]{inp-AndrSchn02}.
The Nichols algebra $\NA(M)$ has the following {\em universal property}:

For any pre-Nichols algebra $R$ of $M$ there is exactly one map
$$\rho : R \to \NA(M),\; \rho \res M = \id,$$
of $\no$-graded braided bialgebras in $\ydH$. Thus $\NA(M)$ is the smallest pre-Nichols algebra of $M$.

In the situation of Theorem \ref{theor:third}, the functor
$$(\Fu,\fu) : \ydRsmashrat \to \ydRveesmashrat$$
is a braided monoidal isomorphism. Hence for any $\no$-graded braided bialgebra $B$ in $\ydRsmashrat$ with multiplication $\mu_B$  and comultiplication $\Delta_B$, the image $\Fu(B)$ is an $\no$-graded braided bialgebra in $\ydRveesmashrat$ with multiplication
$$\Fu(B) \ot \Fu(B) \xrightarrow{\fu_{B,B}} \Fu(B \ot B) \xrightarrow{\Fu(\mu_B)} \Fu(B)$$
and comultiplication
$$\Fu(B) \xrightarrow{\Fu(\Delta_{B})} \Fu(B \ot B) \xrightarrow{\fu_{B,B}^{-1}}\Fu(B) \ot \Fu(B).$$
The unit elements and the augmentations in $B$ and $\Fu(B)$ coincide.

\begin{corol}\label{cor:Nichols}
Under the assumptions of Theorem \ref{theor:third}, let $M \in \ydRsmashrat$. Then
$$\Omega(\NA(M)) \cong \NA(\Omega(M))$$
as $\no$-graded braided Hopf algebras in $\ydRveesmashrat$.
\end{corol}
\begin{proof}
By (N3) and \eqref{pair8}, $\NA(M)$ is rational as an $R$-module, since $M$ is rational. By Theorem \ref{theor:third}, $(\Fu,\fu)$ is a braided monoidal isomorphism. Hence $\NA(M)$ is an $\no$-graded braided bialgebra  in $\ydRsmashrat$. Since $\Fu$ is the identity on morphisms, (N1) -- (N4) hold for $\Fu(\NA(M))$. This proves the Corollary.
\end{proof}

Let $B$ be a coalgebra. An $\ndN _0$-filtration $\cF = (\cF_n B)_{n \in \no}$ of $B$
is a family of subspaces $\cF_nB, n \geq 0,$ of $B$ such that
\begin{itemize}
  \item []$\cF _nB$ is a subspace of $\cF _mB$ for all $m,n\in \ndN _0$ with $n\le m$,
  \item []$B=\bigcup _{n\in \ndN _0}\cF _nB $, and
  \item []$\Delta _B(x)\in \sum _{i=0}^n\cF _iB\otimes \cF _{n-i}B$ for all $x\in \cF _nB$,
    $n\in \ndN _0$.
\end{itemize}

\begin{lemma}\label{lem:F0}
  Let $B$ be a coalgebra having an $\ndN _0$-filtration $\cF $.
  Let $U\in {}^B\cM $ be a non-zero object.
  Then there exists $u\in U\setminus \{0\}$ such that
  $\delta (u)\in \cF_0B \otimes U$.
  \label{le:YDBNH}
\end{lemma}

\begin{proof}
  The coradical $B_0$ of $B$ is contained in $\cF_0B$ by \cite[Lemma 5.3.4]{b-Montg93}. Hence $\delta^{-1}(\cF_0B \ot U) \neq0$, since for any irreducible subcomodule $U'\sub U$ there is a simple subcoalgebra $C'$ with $\delta(U') \sub C'\ot U'$.

  We give an alternative and more explicit proof. Let $x\in U\setminus \{0\}$. Then there exists $n\in \ndN _0$ with $\delta (x)\in \cF _nB\otimes U$.
  If $n=0$, we are done. Assume now that $n\ge 1$ and let $\pi _0:B\to B/\cF _0B$ be the canonical
  linear map.
  Since $\cF $ is a coalgebra filtration, there is a maximal $m\in \ndN _0$ such that
  \[ \pi _0(x_{(-m)})\otimes \cdots \otimes \pi _0(x_{(-1)})\otimes x_{(0)}\not=0, \]
  where $\delta (x)=x\_{-1}\otimes x\_0$.
  Let $f_1,\dots,f_m\in B^*$ with $f_i|_{B_0}=0$ for all $i\in \{1,\dots,m\}$
  such that
  \[ y:=f_1(x_{(-m)})\cdots f_m(x_{(-1)})x_{(0)}\not=0. \]
  Then $\delta (y)=f_1(x_{(-m-1)})\cdots f_m(x_{(-2)})x_{(-1)}\otimes x_{(0)}
  \in \cF _0B \otimes U$ by the maximality of $m$.
\end{proof}

\begin{lemma}\label{lem:ZgradK}
  Let $\Gamma $ be an abelian group with neutral element $0$, and $A$ a $\Gamma $-graded Hopf algebra.
  \begin{enumerate}
  \item Let $K$ be a Nichols algebra in $\ydA $, and $K(1)=\oplus _{\gamma \in \Gamma }K(1)_\gamma $
  a $\Gamma $-graded object in $\ydA $.
  Then there is a unique $\Gamma$-grading on $K$ extending the grading on $K(1)$. Moreover, $K(n)$ is $\Gamma$-graded in $\ydA$ for all $n \geq 0$.
  \item Let $K$ be a $\Gamma $-graded braided Hopf algebra in $\ydA $. Then the bosonization $K \# A$ is a $\Gamma$-graded Hopf algebra with $\deg K(\gamma) \# A(\lambda) = \gamma + \lambda$ for all $\gamma, \lambda \in \Gamma$.
\item Let $H \sub A$ be a Hopf subalgebra of degree {0}, and $\pi : A \to H$ a Hopf algebra map with $\pi \res H = \id$. Define $R= A^{\co H}$. Then $R$ is a $\Gamma$-graded braided Hopf algebra in $\ydH$ with $R(\gamma) = R \cap A(\gamma)$ for all $\gamma \in \Gamma$.
\end{enumerate}
\end{lemma}
\begin{proof}
  (1) The module and comodule maps of $K(1)$ are $\Gamma $-graded and hence the infinitesimal
  braiding $c\in \Aut (K(1)\otimes K(1))$, being determined by the module and comodule maps,
  is $\Gamma $-graded.
  Now the claim of the lemma follows from the fact that
  $K(n)$ for $n\in \ndN $ as well as
  the structure maps of $K$ as a braided Hopf algebra
  are determined by $c$ and $K(1)$.

(2) and (3) are easily checked.
\end{proof}

We now study the projection of $H$-Yetter-Drinfeld Hopf algebras in Corollary \ref{corol:Hopfproj} in the case of Nichols algebras. Recall that for any $M,N\in \ydH $ there is a canonical surjection
$$\pi_{\NA(N)} : \NA(M \oplus N) \to \NA(N),\; \pi_{\NA(N)} \res N = \id, \pi_{\NA(N)} \res M = 0,$$
of braided Hopf algebras in $\ydH$.
It defines a canonical projection
$$\pi _{\NA (N) \# H} = \pi_{\NA(N)} \# \id :\NA (M\oplus N) \# H \to \NA (N) \# H$$
of  Hopf algebras.
Let $K =(\NA(M \oplus N) \# H)^{\co \NA(N) \# H}$ be the space of right $\NA(N) \# H$-coinvariant elements with respect to the projection $\pi_{\NA(N) \# H}$. Thus $K$ is a braided Hopf algebra in $\ydBNH$ with $\NA(N) \# H$-action
$$\ad : \NA(N) \# H \ot K \to K,\; a \ot x \mapsto (\ad a)x= a\sw1 x \cS(a\sw2),$$
and $\NA(N) \# H$-coaction
$$\delta_K : K \to \NA(N) \# H \ot K, x \mapsto \pi_{\NA(N) \# H}(x\sw1) \ot x\sw2.$$
Then by ~\cite[Lemma 3.1]{a-AHS10}, $K = \NA(M \oplus N)^{\co \NA(N)}$,
the space of right $\NA(N)$-coinvariant elements with respect to $\pi_{\NA (N)}$.

The bosonization $\NA(N) \# H$ is a $\mZ$-graded Hopf algebra with $\deg N =1$ and $\deg H = 0$. We always view the bosonizations of Nichols algebras in $\ydH$ as graded Hopf algebras in this way.

\begin{lemma}\label{lem:old}
 Let $M,N\in \ydH$ and
$K=(\NA (M \oplus N) \# H)^{\co \NA (N) \# H}$.
\begin{enumerate}
\item The standard $\no$-grading of $\NA(M \oplus N)$ induces an $\no$-grading on
$$W = (\ad \NA(N))(M) = \oplus_{n \in \no} (\ad N)^n(M)$$ with $\deg (\ad N)^n(M)=n +1$. Then $W$ is a $\mZ$-graded object in $\ydBNH$, where $W \sub K$ is a subobject in $\ydBNH$.
\item Assume that $M = \oplus_{i \in I} M_i$ is a direct sum of irreducible objects in $\ydH$. Let $W_i = (\ad\NA(N))(M_i)$ for all $i \in I$. Then $W = \oplus_{i \in I} W_i$ is a decomposition into irreducible subobjects $W_i$ in $\ydBNH$. For all $i \in I$, $W_i = \oplus_{n \geq 0}(\ad N)^n(M_i)$ is a $\ndZ$-graded object in the category of left Yetter-Drinfeld modules over $\NA(N) \# H$.
\end{enumerate}
\end{lemma}
\begin{proof}
(1) Let $a \in N$ and $x \in \NA(M \oplus N)$ a homogeneous element. Then $\Delta_{\NA(M \oplus N) \# H}(a) = a \ot 1 + a\sw{-1} \ot a\sw0$, since $a$ is primitive in $\NA(N)$. Hence
$$(\ad a)(x) = ax - (a\sw{-1} \lact x) a\sw0$$
is of degree $\deg x + 1$ in $\NA(M \oplus N)$. This implies the decomposition of $W$. Moreover, $W \sub K$, since $M \sub K$.

Since $W= (\ad \NA(N) \# H)(M)$, it is clear that $W$ is stable under the adjoint action of $\NA(N) \# H$, and that
$$\ad : \NA(N) \# H \ot W \to W$$
is $\mZ$-graded. To see that $W \sub K$ is a $\NA(N) \# H$-subcomodule, and that the comodule structure
$$W \to \NA(N) \# H \ot W$$
is $\mZ$-graded, we compute $\delta_K$ on elements of $W$. For all $a \in \NA(N) \# H$ and $x \in M$,
\begin{align*}
\delta_K(\ad a)(x) &= (\pi_{\NA(N) \# H} \ot \id)\Delta_{\NA(M \oplus N) \# H}(\ad a)(x)\\
&=\pi_{\NA(N) \# H}\left(a\sw1 x\sw1 \cS(a\sw4)\right) \ot a\sw2 x\sw2 \cS(a\sw3)\\
&=\pi_{\NA(N) \# H}\left(a\sw1 x \cS(a\sw4)\right) \ot a\sw2  \cS(a\sw3)\\
&\phantom{aa}+ \pi_{\NA(N) \# H}\left(a\sw1 x\sw{-1} \cS(a\sw4)\right) \ot a\sw2 x\sw0 \cS(a\sw3)\\
&= a\sw1 x\sw{-1} \cS(a\sw3) \ot (\ad a\sw2)(x\sw0).
\end{align*}
Thus the $\NA(N) \#H$-costructure of $W$ is well-defined and $\mZ$-graded.

(2) is shown in ~\cite[Prop. 3.4, Prop. 3.5]{a-AHS10}.
\end{proof}

\begin{propo}\label{prop:old}
Let $M,N\in \ydH$ and
$K=(\NA (M \oplus N) \# H)^{\co \NA (N) \# H}$. Then there is a unique isomorphism
$$K\cong \NA \big( (\ad \NA (N))(M) \big)$$
  of braided Hopf algebras in $\ydBNH$ which is the identity on $(\ad \NA (N))(M)$.
\end{propo}
\begin{proof}
 Since $M \oplus N$ is a $\mZ$-graded object in $\ydH$ with $\deg M = 1$ and $\deg N =0$, the Nichols algebra $\NA(M \oplus N)$ is a $\mZ$-graded braided Hopf algebra in $\ydH$ by Lemma \ref{lem:ZgradK} (1). Hence the bosonization $\NA(M \oplus N) \# H$ is a $\mZ$-graded Hopf algebra with $\deg M =1, \deg N =0, \deg H =0$. By Lemma \ref{lem:ZgradK} (3), $K$ is a $\mZ$-graded Hopf algebra in $\ydBNH$. By ~\cite[Prop. 3.6]{a-AHS10}, $K$ is generated as an algebra by $K(1)= (\ad \NA(N))(M)$. Hence $K(n) = K(1)^n$ for all $n \geq 1$, and $K(0) = \fie$.

 It remains to prove that all homogeneous primitive elements of $K$ are of
 degree one. Let $n\in \ndN _{\ge 2}$ and let $U\subseteq K(n)$ be a subspace
 of primitive elements. We have to show that $U=\{0\}$.
 By Remark~\ref{rem:primitive} (1) we may assume that $U\in \ydBNH $.
 Since $\NA (N)\#H$ has a coalgebra filtration $\cF $ with $\cF _0=H$ and $\cF _1=NH+H$,
 Lemma~\ref{le:YDBNH} implies that there exists a nonzero primitive element $u\in U$ with
 $\delta (u)\in H\otimes U$. Then $u$ is primitive in $\NA (M \oplus N)$. Indeed,
 $$ \Delta _{K\#(\NA (N)\#H)}=u\otimes 1+1u_{[-1]}\otimes u_{[0]}
    =u\otimes 1+u_{(-1)}\otimes u_{(0)},
 $$
 and hence $\Delta _{K\#\NA (N)}(u)=(\vartheta \otimes \id )\Delta _{K\#(\NA (N)\#H)}(u)
 =u\otimes 1+1\otimes u$.

 Since $K(n) = \left(\ad \NA(N)(M)\right)^n$,
 $u$ is an element of degree at least $n$ in the usual grading of $\NA(M \oplus N)$.
 This contradicts the assumption that
 $\NA (M \oplus N)$ is a Nichols algebra.
\end{proof}

Next we prove the converse of the above proposition under additional restrictions, see
Proposition~\ref{propo:KBNichols}.

Let $C$ be a coalgebra, $D \sub C$ a subcoalgebra, and $W$ a left $C$-comodule with comodule structure $\delta : W \to C \ot W$.
We denote  the largest $D$-subcomodule of $W$ by
\begin{align*}
W(D) = \{ w \in W \mid \delta(w) \in D \ot W\}.
\end{align*}

\begin{lemma}
  Let $N\in \ydH $ and $W\in \ydBNH $.
  Assume that $\oplus _{i\in I}W_i$ is a decomposition of $W$ into irreducible objects
  in the category of $\ndZ$-graded left Yetter-Drinfeld modules over $\NA(N) \# H$.
   Let $M=W(H)$, and $M_i=M\cap W_i$ for all $i \in I$.
  \begin{enumerate}
    \item $M=\oplus _{i\in I}M_i$ is a decomposition into irreducible
      objects in $\ydH $.
    \item For all $i\in I$, $M_i$ is the $\ndZ $-homogeneous component
      of $W_i$ of minimal degree, and $W_i=B(N)\cdot M_i = \oplus_{n \geq 0} N^n \lact M_i$.
  \end{enumerate}
  \label{lem:ydBNH}
\end{lemma}

\begin{proof}
  Let $W= \oplus_{n \in \mZ} W(n)$ be the $\mZ$-grading of $W$ in $\ydBNH $. Then
  $M$ is a $\mZ$-graded object in $\ydH$ with homogeneous components $M(n) = M \cap W(n)$ for all $n \in \mZ$.
It is clear that $M = \oplus_{i \in I} M_i$, where $M_i = M \cap W_i = W_i(H)$ for all $i$.

Let $i \in I$. By Lemma \ref{lem:F0}, $M_i \neq 0$.
  Let $0 \neq M'_i$ be a homogeneous subobject of $M_i$ in $\ydH $, and let $n$ be its degree.
  Then the $\NA (N)\#H$-module $W'_i:=\NA (N)\cdot M'_i$ is a $\ndZ $-graded
  subobject of $W_i$ in $\ydBNH $,
  the homogeneous components of $W'_i$ have degrees $\ge n$,
  and the degree $n$ component of $W'_i$ coincides with $M'_i$ since $\NA (N)(0)=\fie $
  and $\deg N = 1$.
  Thus the irreducibility of $W_i$ implies that $M_i=M'_i$ is irreducible and it is the
  homogeneous component of $W_i$ of minimal degree.

Finally, for all $i \in I$ and $n \in \no$,
  $$\deg (N^n \lact M_i) = n + \deg M_i,$$
  since the multiplication map $\NA(N) \# H \ot W_i \to W_i$ is graded. It follows that $W_i = \oplus_{n \geq 0}N^n \lact M_i$ for all $i$.
\end{proof}

\begin{propo}
  Let $N\in \ydH$ and $W\in \ydBNH$. Assume that $W$ is a semisimple object in the category of $\ndZ$-graded left Yetter-Drinfeld modules over $\NA(N) \# H$. Let $K=\NA (W)$ be the Nichols algebra
  of $W$ in $\ydBNH $, and define $M = W(H)$.
  Then there is a unique isomorphism
  $$K\#\NA (N)\cong \NA (M \oplus N)$$
   of braided Hopf algebras in $\ydH $ which is the identity on $M \oplus N$.
  \label{propo:KBNichols}
\end{propo}

\begin{proof}
  Let $\oplus _{i\in I}W_i$ be a decomposition of $W$ into irreducible objects in the category of $\ndZ$-graded left Yetter-Drinfeld modules over $\NA(N) \# H$. For all $i \in I$, let $M_i = W_i \cap M$. By Lemma~\ref{lem:ydBNH} (2), we can define a new $\mZ$-grading on $W$ by
  $$\deg (N^n\cdot M_i)=n+1$$
  for all $n\in \ndN _0, i \in I$. Then $W$ is a $\mZ$-graded object in $\ydBNH$. Because of Lemma~\ref{lem:ZgradK} (1),
  and since $W=K(1)$, we know that $K$ is a $\ndZ $-graded braided Hopf algebra with this new
  $\ndZ $-grading on $K(1)$. Thus by Lemma~\ref{lem:ZgradK} (2) and (3), $K\#(\NA (N) \# H)$ is a $\ndZ $-graded Hopf algebra, and
  $$R:=K\#\NA (N)= \left(K\#(\NA (N) \# H)\right)^{\co H}$$
  is a $\ndZ $-graded braided Hopf algebra
  in $\ydH $
  with $\fie 1$ as degree $0$ part and with
  $M \oplus N$ as degree $1$ part.

  Let $m \in M$ and $b \in \NA(N)$. Then
  \begin{align}\label{ruleR}
  b \lact m = b\swo1 ({b\swo2}\sw{-1} \lact m) \cS_{\NA(N)}({b\swo2}\sw0)
  \end{align}
  in the algebra $R=K\#\NA (N)$.
  Since $K$ is generated as an algebra by $K(1)$, and since
  $K(1)=\NA (N)\cdot M$, we conclude from \eqref{ruleR} that
  $R$ is generated as an algebra by $R(1)=M \oplus N$. Thus $R$ is a  pre-Nichols algebra of $M \oplus N$.

  By the universal property of the Nichols algebra $\NA(M \oplus N)$, there is a surjective homomorphism
  $$\rho : R \to \NA(M \oplus N),\; \rho \res M \oplus N = \id,$$
  of $\no$-graded Hopf algebras in $\ydH$. Then
  $$ \rho \# \id : R \# H \to \NA(M \oplus N) \# H$$
  is a surjective map of Hopf algebras. Let $K'= (\NA(M \oplus N) \# H)^{\co \NA(N) \# H}$. Since the multiplication maps
  $$R \# H \to K \# (\NA(N) \# H),\; K' \# (\NA(N) \# H) \to \NA(M \oplus N) \# H$$
  are bijective maps of Hopf algebras, the map $\rho \# \id$ defines a surjective map of Hopf algebras
  $$\rho': K \# (\NA(N) \# H) \to K'\# (\NA(N) \# H),\; \rho'\res (M \oplus N) = \id.$$
  The action of $\NA(N) \# H$ on $K$ is the adjoint action in $K \# (\NA(N) \# H)$.
  Since the algebras $K$ and $K'$ are generated by $(\ad \NA(N))(M)$ on both sides, $\rho'$ induces a map
  $$\rho_1 : K \to K', \; \rho_1 \res M = \id,$$
  of $\no$-graded braided Hopf algebras in $\ydBNH$, and a map
$$\rho_2 : \NA(N) \lact M \to (\ad\NA(N))(M), \; \rho_2 \res M = \id,$$
  in $\ydBNH$. Since $(\ad\NA(N))(M_i)$ is irreducible in $\ydBNH$ for all $i \in I$, it follows that $\rho_2$ is bijective. Hence $\rho_1$ is bijective by the universal property of the Nichols algebra $K = \NA(W)$. Thus $\rho = \rho_1 \# \id_{\NA(N)}$ is bijective.
  \end{proof}

We  now apply Corollary \ref{cor:third} to Nichols algebras.
Let $N \in \ydH$ be finite-dimensional. Then the dual vector space $N^* = \Hom(N,\fie)$ is an object in $\ydH$ with
\begin{align*}
\langle h \lact \xi,x\rangle &= \langle \xi,\cS(h) \lact x\rangle,\\
\xi\sw{-1} \langle\xi\sw0,x\rangle &= \cS^{-1}(x\sw{-1}) \langle \xi,x\sw0\rangle
\end{align*}
for all $\xi \in N^*, x \in N, h \in H$, where $\langle \;,\; \rangle : N^* \ot N \to \fie$ is the evaluation map. The Nichols algebras of the finite-dimensional Yetter-Drinfeld modules $N^*$ and $N$ have finite-dimensional $\no$-homogeneous components, and there is a canonical pairing
$\langle \;,\; \rangle : \NA(N^*) \ot \NA(N) \to \fie$
extending the evaluation map such that the conditions \eqref{pair1} -- \eqref{pair5} and \eqref{gradedpair} hold, see for example \cite[Prop.\, 1.10]{a-AHS10}. Let
$$(\Fu_N,\fu_N) : \ydBNsmashrat \to \ydBNdualsmashrat$$
be the functor of Corollary \ref{cor:third} with respect to the canonical dual pairing.

\begin{theor}
  \label{theor:NtoN}
  Let $n \geq 1$, and let $M_1,\dots, M_n,N$ be finite-dimensional objects in $\ydH$. Assume that for all $1 \leq i \leq n$, $M_i$ is  irreducible in $\ydH$, and
  that $(\ad \NA(N))(M_i)$ is a finite-dimensional subspace of $\NA (\oplus_{i = 1}^n M_i \oplus N)$. For all $i$
  let $V_i = \Fm 0(\ad \NA (N))(M_i)$, and let
  $K=\NA (\oplus_{i = 1}^n M_i \oplus N)^{\co \NA(N)}$.
  \begin{enumerate}
  \item The modules $V_1,\dots, V_n$ are irreducible  in $\ydH$, $\Fu_N(K)$ is a braided Hopf algebra in $\ydBNdualsmashrat$, and there is a unique isomorphism
  $$\Fu_N(K)\#\NA (N^*)\cong \NA (\oplus_{i=1}^n V_i\oplus N^*)$$
  of braided Hopf algebras in $\ydH $ which is the identity on $\oplus_{i=1}^n V_i\oplus N^*$.
 \item For all $1 \leq i \leq n$, let $m_i = \max \{m \in \no \mid (\ad N)^m(M_i) \neq 0 \}$, and $W_i = (\ad \NA(N))(M_i)$. Then
 \begin{align*}
 W_i &= \mathop\oplus_{n=0}^{m_i} (\ad N)^n(M_i),& V_i &= (\ad N)^{m_i}(M_i),\\
 \Fu_N(W_i) &\cong \mathop\oplus_{n=0}^{m_i} (\ad N^*)^n(V_i),& M_i &\cong (\ad N^*)^{m_i}(V_i)
 \end{align*}
 for all $i$.
  \end{enumerate}
\end{theor}
\begin{proof}
(1) Let $W = (\ad \NA(N))(M)$.  By Lemma \ref{lem:old} (2), $W_1,\dots,W_n$ are
irreducible objects in $\ydBNH$, $W = \oplus_{i=1}^n W_i$, and for all
$1 \leq i \leq n$, $M_i$ is the $\mZ$-homogeneous component of $W_i$ of minimal degree.
By Proposition~\ref{propo:Fdm}, the Yetter-Drinfeld modules $V_1,\dots,V_n\in \ydH $ are
irreducible.
By Proposition \ref{prop:old}, $K$ is isomorphic to the Nichols algebra of $W$ in
$\ydBNH$. Since $(\ad \NA(N))(M)$ is a finite-dimensional and $\mZ$-graded object
in $\ydBNH$, it is a rational  $\NA(N)$-module.
Therefore $\Fu_N(\NA(W)) \cong \NA(\Fu_N(W))$ by Corollary \ref{cor:Nichols}.

Hence there is a unique isomorphism $\Fu_N(K) \cong \NA(\Fu_N(W))$ of braided Hopf algebras in $\ydBNdualsmashrat$ which is the identity on $\Fu_N(W)$. Recall that
$$\Fu_N(W)(H) = \Fd0 \Fu_N(W) = \Fm0 W = \oplus_{i=1}^n V_i$$
by Corollary \ref{cor:third} (3). Then by Proposition \ref{propo:KBNichols} there is a unique isomorphism
$$\Fu_N(K) \# \NA(N^*) \cong \NA (\oplus_{i=1}^n V_i\oplus N^*)$$
 of braided Hopf algebras in $\ydH$ which is the identity on $\oplus_{i=1}^n V_i\oplus N^*$ which proves (1). For the last conclusion we have to check the assumptions of Proposition \ref{propo:KBNichols}, that is, $\Fu_N(W) \in \ydBNdualsmashrat$ is a semisimple $\ndZ$-graded Yetter-Drinfeld module. By Corollary \ref{cor:third} (2), $\Fu_N(W)$ is $\ndZ$-graded, and it is semisimple  since $W$ is semisimple by Lemma \ref{lem:old} and $\Fu_N$ is an isomorphism by Corollary \ref{cor:third}.

(2) Let $i\in \{1,\dots,n\}$. The first equation follows from the definition of $W_i$
and the second from Proposition~\ref{propo:Fdm} for $R=\NA (N)$ with $\deg N=1$ and
$\deg (\ad N)^n(M_i)=1+n$ for all $n\ge 0$.
By Corollary~\ref{cor:third}, $\Fu _N(W_i)=W_i$ is $\ndZ $-graded with homogeneous components
$(\ad N)^n(M_i)$ of degree $-n-1$. Moreover, $V_i=\Fd 0\Fu _N(W_i)$ by the proof of (1),
and hence $\Fu _N(W_i)=\oplus _{n=0}^{m_i}(N^*)^nV_i$ since $\Fu _N(W_i)$
is irreducible. In particular, $M_i=(N^*)^{m_i}V_i$. These equations imply
the remaining claims of (2).
\end{proof}

\begin{remar}
Theorem \ref{theor:NtoN} still holds if we replace the canonical pairing in the definition of $(\Fu_N,\fu_N)$ by any dual pairing
$\langle \;,\; \rangle : \NA(N^*) \ot \NA(N) \to \fie$
satisfying \eqref{pair1} -- \eqref{pair5} and \eqref{gradedpair}.
\end{remar}

The definition of
the Weyl groupoid of a Nichols algebra of a semisimple Yetter-Drinfeld module over $H$
is based on
\cite[Thm.\,3.12]{a-AHS10},
see also \cite[Sect.\,3.5]{a-AHS10} and \cite[Thm.\,6.10, Sect.\,5]{a-HeckSchn10}.
To see that Theorem~\ref{theor:NtoN} can be considered as an alternative approach to the
definition of the Weyl groupoid, we introduce some notations.

Let $\theta \geq 1$ be a natural number. Let $\cF_{\theta}$ denote the class of all families
$M = (M_1,\dots,M_{\theta})$, where $M_1,\dots,M_{\theta} \in \ydH$
are finite-dimensional irreducible Yetter-Drinfeld modules. If $M \in \cF_{\theta}$, we define
$$\NA(M)=\NA(M_1 \oplus \cdots \oplus M_{\theta}).$$
For families $M,M'\in \cF_\theta $, we write $M \cong  M'$, if $M_j \cong M_j'$ in $\ydH $ for all $j$.

For $1 \leq i \leq \theta$ and $M \in \cF_{\theta}$, we say that the
{\em $i$-th reflection $R_i(M)$ is defined} if for
all $j \neq i$ there is a natural number $m_{ij}^M \geq 0$ such that
$(\ad M_i)^{m_{ij}^M}(M_j)$ is a non-zero finite-dimensional subspace
of $\NA(M)$, and $(\ad M_i)^{m_{ij}^M+1}(M_j)=0$. Assume that $R_i(M)$ is defined.
Then we set $R_i(M) = (V_1,\dots,V_{\theta})$, where
\begin{align*}
V_j &= \begin{cases}
V_i^*, &\text{ if } j=i,\\
(\ad M_i)^{m^M_{ij}}(M_j), &\text{ if } j \neq i,
\end{cases}
\end{align*}
For all $j \neq i$, let $a_{ij}^M = -m_{ij}^M$.
By \cite[Lemma\,3.22]{a-AHS10}, $(a_{ij}^M)$ with $a_{i i}^M=2$ for all $i$
is a generalized Cartan matrix.

The next Corollary follows from a restatement of Theorem \ref{theor:NtoN}. Thus we obtain a new proof of \cite[Thm.\,3.12(2)]{a-AHS10} which allows to define the Weyl groupoid of $M \in \cF_{\theta}$.

 \begin{corol} \cite[Thm.\,3.12(2)]{a-AHS10}
Let  $M \in \cF_{\theta}$, and $1 \leq i \leq \theta$. Assume that $R_i(M)$ is defined. Then
$R_i(M) \in \cF_{\theta}$, $R_i^2(M)$ is defined, $R_i^2(M) \cong M$, and $a_{ij}^M = a_{ij}^{R_i(M)}$ for all $1 \leq j \leq \theta$.
\end{corol}
In the situation of the last Corollary, let $K_i^M = \NA(M)^{\co \NA(M_i)}$ with respect to the projection $\NA(M) \to \NA(M_i)$. Then
$$K_i^M \# \NA(M_i) \cong \NA(M)$$
by bosonization, and
$$\Fu(K_i^M) \# \NA(M_i^*) \cong \NA(R_i(M))$$
by Theorem \ref{theor:NtoN}.


\providecommand{\bysame}{\leavevmode\hbox to3em{\hrulefill}\thinspace}
\providecommand{\MR}{\relax\ifhmode\unskip\space\fi MR }
\providecommand{\MRhref}[2]{%
  \href{http://www.ams.org/mathscinet-getitem?mr=#1}{#2}
}
\providecommand{\href}[2]{#2}

\end{document}